\documentclass[11pt]{article}
\usepackage[margin=1in]{geometry}
\usepackage{amssymb,amsfonts,amsmath,bbm,mathrsfs,stmaryrd,mathtools}
\usepackage{xcolor}
\usepackage{url}

\usepackage[shortlabels]{enumitem}
\SetEnumitemKey{(1)}{label=(\arabic*), ref=\arabic*}
\SetEnumitemKey{(a)}{label=(\alph*), ref=\alph*}
\setlist[enumerate]{wide=0pt}
\setlist[itemize]{wide=0pt}

\usepackage[backref=page]{hyperref}
\hypersetup{colorlinks,
             linkcolor=black!75!red,
             citecolor=blue,
             pdftitle={},
             pdfproducer={pdfLaTeX},
             pdfpagemode=UseNone,
             bookmarksopen=true,
             bookmarksnumbered=true}

\usepackage{tikz}
\usetikzlibrary{arrows,calc,decorations.pathreplacing,decorations.markings,intersections,shapes.geometric,through,fit,shapes.symbols,positioning,decorations.pathmorphing}

\usepackage{braket}
\usepackage{tensor}

\usepackage[amsmath,thmmarks,hyperref]{ntheorem}

\theoremstyle{plain}

\newtheorem{lemma}{Lemma}[section]
\newtheorem{proposition}[lemma]{Proposition}
\newtheorem{corollary}[lemma]{Corollary}
\newtheorem{theorem}[lemma]{Theorem}

\theoremstyle{plain}
\theoremnumbering{Alph}
\newtheorem{theoremN}{Theorem}

\theoremstyle{plain}
\theorembodyfont{\upshape}
\theoremsymbol{\ensuremath{\blacklozenge}}

\newtheorem{definition}[lemma]{Definition}
\newtheorem{example}[lemma]{Example}
\newtheorem{examples}[lemma]{Examples}
\newtheorem{remark}[lemma]{Remark}
\newtheorem{remarks}[lemma]{Remarks}

\numberwithin{equation}{section}

\theoremstyle{nonumberplain}
\theoremsymbol{\ensuremath{\blacksquare}}

\newtheorem{proof}{Proof}

\newcommand\bR{{\mathbb R}}

\newcommand\bZ{{\mathbb Z}}

\newcommand\cF{{\mathcal F}}

\newcommand\cO{{\mathcal O}}

\newcommand\cX{{\mathcal X}}
\newcommand\cY{{\mathcal Y}}

\DeclareMathOperator{\id}{id}

\DeclareMathOperator{\Aut}{\mathrm{Aut}}

\newcommand{\qedhere}{\mbox{}\hfill\ensuremath{\blacksquare}}

\title{P\l onka bi-magmas and the set-theoretic Yang--Baxter equation}

\author{A.L. Agore, A. Chirvasitu and G. Militaru}

\begin{document}

\date{}

\newcommand{\Addresses}{{
  \bigskip
  \footnotesize

\author{A. L. Agore:}
 \textsc{Simion Stoilow Institute of Mathematics of the Romanian Academy}
  \par\nopagebreak
  \textsc{Bucharest, P.O. Box 1-764, 014700, Romania}
  \par\nopagebreak
  \textit{E-mail address}: \texttt{ana.agore@gmail.com}

  \medskip

  \author{A. Chirvasitu:}
  \textsc{Department of Mathematics, University at Buffalo}
  \par\nopagebreak
  \textsc{Buffalo, NY 14260-2900, USA}
  \par\nopagebreak
  \textit{E-mail address}: \texttt{achirvas@buffalo.edu}

  \medskip

\author{G. Militaru:}
  \textsc{Faculty of Mathematics and Computer Science, University of Bucharest}
  \par\nopagebreak
  \textsc{Str. Academiei 14, RO-010014 Bucharest 1, Romania}
  \par\nopagebreak
  \textit{E-mail address}: \texttt{gigel.militaru@gmail.com}

}}

\maketitle

\begin{abstract}
We introduce a new variety of universal algebras with two binary operations, called \emph{P{\l}onka bi-magmas}, and use them to construct novel solutions to the set-theoretic Yang-Baxter (YB) equation. We establish explicit structural and classification results for P{\l}onka bi-magmas, which allow for a detailed study of the induced families of YB-solutions. In particular, we identify and classify several naturally arising classes of YB-solutions from the universal algebra perspective, including the so-called bi-connected and ideal-simple ones. For instance, we prove that there are only countably many isomorphism classes of induced YB-solutions which are ideal-simple, and characterize them in terms of the odometer transformations familiar from ergodic theory. In particular, if $X$ is a finite set with $|X| = n$, then the number of isomorphism classes of all the induced ideal-simple YB-solutions on $X$ is equal to the sum of all divisors of $n$.
\end{abstract}


\vspace{.5cm}



\section*{Introduction}

Over the past half century, the celebrated quantum Yang--Baxter equation (also known as the 2-simplex or triangle equation) has played a pivotal role in mathematical physics and has significantly shaped many branches of pure mathematics. Its broad reach encompasses the field of solvable models of statistical mechanics (where it originated, in the independent work of Yang \cite{Yang} and Baxter \cite{baxter}), the theory of quantum groups, invariants of knots and three-dimensional manifolds, braided monoidal categories and conformal field theory. Furthermore, its study has had a significant impact on group theory, combinatorics and non-commutative algebraic geometry.\\
Recall that a linear map $R: V\otimes V \to V\otimes V$ is a solution of the \emph{quantum Yang-Baxter equation} if
\begin{equation*}
R^{12} R^{13} R^{23} = R^{23}R^{13}R^{12}
\end{equation*}
under the usual map composition in the endomorphism algebra ${\rm End} (V\otimes V \otimes V)$, where $V$ is a vector space over a field $k$ and $R^{ij} : V\otimes V \otimes V \to V\otimes V \otimes V$ acts as $R$ on the $i$-th and $j$-th tensor factors and as the identity on the remaining factor.\\
Beyond its purely theoretical importance, the problem of classifying all Yang-Baxter solutions for a given finite dimensional vector space $V$ is central to constructing integrable models in mathematical physics. In full generality, however, the problem is notoriously difficult: it has up to this point been completely solved only in dimension $2$ by Hietarinta \cite{hiet_1} whereas in dimension $3$ only partial classifications have been obtained using computer methods (see e.g. \cite{hiet_2}). Notably, Radford \cite{radford_1} approached the problem algebraically by building on the well-known FRT-theorem of Faddeev, Reshetikhin and Takhtajan \cite{FRT}.\\
These difficulties prompted Drinfel'd \cite[Section 9]{dr} to adopt a different perspective by proposing to restrict to the classification of those YB-solutions arising from a given {\it basis} $X$ of the vector space $V$: in other words, studying the Yang-Baxter equation at the level of sets. Given a set $X$, a map $R: X\times X \to X\times X$ is a solution of the \emph{Yang-Baxter equation} (resp. \emph{braid equation}) if
\begin{equation}\label{eq:qybe}
R^{12} R^{13} R^{23} = R^{23}R^{13}R^{12}, \qquad  ({\rm resp.} \,\, R^{12}R^{23}R^{12} = R^{23}R^{12}R^{23})
\end{equation}
as maps $X\times X \times X \to X\times X \times X$ under the usual composition, where $R^{12} = R \times {\rm Id}_X$, $R^{23} = {\rm Id}_X \times R$ and $R^{13} = ({\rm Id}_X \times \tau_X) \circ R^{12} \circ ({\rm Id}_X \times \tau_X)$, with $\tau_X$ being the flip map. The above two equations are equivalent: $R$ is a solution of the braid equation if and only if $\tau_X \circ R$ is a solution of the Yang-Baxter equation, so we can focus on studying just one of them.
Throughout, we work with the Yang-Baxter equation and a map $R: X\times X \to X\times X$ satisfying the Yang-Baxter equation will be referred to as a \emph{YB-solution}. Such solutions are not, a priori, required to satisfy any additional constraints. However, in concrete classification problems it is both natural and desirable to focus on particular subclasses of solutions.\\
Even in its set-theoretic form, the problem has proved to be highly challenging and has led to a wealth of mathematical developments (see \cite{BES, CV} and the references therein). An early contribution to the problem was made by Weinstein and Xu \cite[\S 7]{WeXu}, who construct several families of YB-solutions arising from Poisson Lie groups. The problem was subsequently studied by Gateva-Ivanova and Van den Bergh \cite{GIVB}, Etingof, Schedler and Soloviev \cite{etingof}, and Lu, Yan and Zhu \cite{LuYZ}. These influential papers laid the foundations for the study of set-theoretic YB-solutions. In \cite{etingof}, two fundamental classes of YB-solutions were introduced: the \emph{non-degenerate} and the \emph{indecomposable} ones. The first successful classification results for these classes of solutions \cite{etingof, etingof2, soloviev} gave rise to an extensive literature devoted almost entirely to their study, often under additional assumptions such as involutivity or unitarity, or with slight variations of the original definitions.\\
The approach we propose here is based on universal algebras whose underlying sets $X$ are equipped with either one or two binary operations, referred to as \emph{magmas} and \emph{bi-magmas}, respectively. Further details and any undefined terminology are provided in Section \ref{se:prel}. In order to explain how the bi-magma setting fits naturally in the study of YB-solutions, recall that for a given set $X$ the following  map is obviously bijective:
$$
{\rm Hom}_{\rm Set} \, (X \times X, \, X \times X) \to {\rm Hom}_{\rm Set} \, (X \times X, \, X) \times {\rm Hom}_{\rm Set} \, (X \times X, \, X), \quad
R \mapsto (\pi_1 \circ R, \, \pi_2 \circ R)
$$
where $\pi_1$, $\pi_2 \colon X\times X \to X$ denote the projection maps. In other words, we have a one-to-one correspondence
between the set of all functions $R: X\times X \to X \times X$ and the set of all bi-magma structures on the underlying set $X$. Furthermore, the bijection is such that the map $R = R_{(\cdot, \, \ast)} : X\times X \to X \times X$ associated to the bi-magma $(X, \, \cdot, \, \ast)$ is given for any $x$, $y\in X$ by:
\begin{equation}\label{notatieR0}
R(x, \, y) = R_{(\cdot, \, \ast)} \, (x, \, y) = (x\cdot y, \, x \ast y)
\end{equation}
where we denote $x \cdot y = (\pi_1 \circ R) (x,y)$ and $x \ast y = (\pi_2 \circ R) (x,y)$.\\
Throughout, rather than considering maps $R : X\times X \to X \times X$ we will often work with bi-magmas $(X, \, \cdot, \, \ast)$ in which case we implicitly assume that $R = R_{(\cdot, \, \ast)}$ is given by \eqref{notatieR0} and refer to it as the \emph{canonical map} associated with the bi-magma $(X, \, \cdot, \, \ast)$. Adopting the bi-magma terminology offers numerous advantages. In particular, it provides a new perspective on well-established concepts from the Yang-Baxter literature. For example, as explained in Remark~\ref{re:nede-quasigr}, the bi-magma formalism reveals that the classical and extensively studied notion of non-degeneracy \cite{etingof} is only one instance of a broader family of non-degeneracy conditions that can be imposed on a YB-solution.\\
Returning to Drinfel'd's problem and writing an arbitrary map $R: X\times X \to X\times X$ as $R (x, \, y) = (x \cdot y, \, x\ast y)$ for some bi-magma $(X, \, \cdot, \, \ast)$, it is easily shown that $R = R_{(\cdot, \, \ast)}$ is a YB-solution if and only if the following conditions hold
\begin{align}
  (x\cdot y) \cdot z &= \bigl(x \cdot (y \ast z) \bigl) \cdot (y \cdot z)\label{YB1a}\\
  \bigl( x \cdot (y \ast z) \bigl) \ast (y\cdot z ) &= (x \ast y) \cdot \bigl( (x\cdot y) \ast z  \bigl)\label{YB2a} \\
  x \ast (y \ast z) &=  (x \ast y) \ast \bigl( (x\cdot y) \ast z\bigl)\label{YB3a}
\end{align}
for all $x$, $y$, $z\in X$ (see e.g. \cite[Lemma 2.1]{rusii}). A triple $(X, \, \cdot, \, \ast)$ satisfying the compatibilities \eqref{YB1a}-\eqref{YB3a} will be called a \emph{Yang-Baxter bi-magma} (Definition \ref{def:ybbimagm}) (or a \emph{set-theoretic non-associative Yang-Baxter algebra}) and the category of all Yang-Baxter bi-magmas is denoted by ${}_{\rm YB}{\rm BiMag}$. With this concept in hand, the problem can be restated as follows:\\
\emph{For a given (finite) set $X$ describe and classify up to isomorphism all Yang-Baxter bi-magmas definable on $X$.}\\
The identities \eqref{YB1a} and \eqref{YB3a} show that the binary operations $\cdot$ and $\ast: X \times X \to X$ are far from being associative which suggests the inherent difficulty of the problem.
Furthermore, the defining equations \eqref{YB1a}-\eqref{YB3a} make it plain that the category of Yang-Baxter bi-magmas is a \emph{variety of algebras} in the sense of \emph{universal algebra} (\cite[\S 3.A]{adamek_ros}). This observation provides a natural motivation for approaching the subject from the perspective of universal algebra. In fact, the link with universal algebra is implicit in the literature \cite{cat_maz_ste, Deho, PPZ, SV}, even though it is not typically made explicit. A perusal of the universal-algebra literature highlights a beautiful structure theorem on a certain class of universal algebras proved by P{\l}onka \cite[Theorem 1]{plonka2}: there are only four types of varieties of algebras having exactly $n$ $n$-ary operations depending on all variables for each positive integer $n \geq 1$. One of these structures is what we refer to as a \emph{$2$-cyclic magma} (or a \emph{$2$-cyclic groupoid} in \cite{plonka}; see Definition \ref{def:2cyc}). Now, a certain \emph{canonical map} associated to any $2$-cyclic magma is a unitary and diagonal YB-solution (Corollary \ref{cor:famI}). \\
Furthermore, from this perspective, it is to be expected that the various group-theoretic constructs familiar from the Yang-Baxter literature over the past $25$ years (\cite[Corollary 3.14]{bai_guo}, \cite[Theorem 2.9]{etingof}, \cite[Theorem 1]{LuYZ}) may overlook large families of solutions. This further motivates the pursuit of alternative perspectives and strategies beyond those currently prevalent in the literature, which may in turn lead to new classes of YB-solutions.\\
\textbf{Outline of the paper.} Section \ref{se:prel} recalls basic concepts used throughout the paper such as universal algebras or varieties of algebras. All universal algebras used in this paper will be sets $X$ equipped with one or two binary operations, referred to, respectively, as {\it magmas} and {\it bi-magmas}. \\
Section \ref{se:sect2a} is somewhat combinatorial in nature. We introduce an important family of functions $R: X\times X \to X\times X$ on sets, called \emph{ideal-simple} (Definition \ref{def:solsimple}) from the viewpoint of universal algebras. In order to classify all ideal-simple Lyubashenko YB-solutions of Example \ref{exs:bmlyb}(\ref{item:LB}) we introduce in Definition \ref{def:simple} the notion of an {\it incompressible} family $\cF$ of commuting self-maps $X\to X$, namely one which leaves no proper non-empty subset of $X$ invariant. Our first results, Theorem \ref{th:allcomm} and Corollary \ref{cor:howmany}, classify all commuting incompressible families of self-maps. As an application, we classify the ideal-simple Lyubashenko YB-solutions: they are  ``almost finite'', and describable by pairs of functions consisting of a cycle and an \emph{odometer} or \emph{adding machine map}, in ergodic-theoretic language \cite[page 266]{np_ergodic}. \\
In Section \ref{se:plonkabi} we introduce \emph{P{\l}onka bi-magmas}, the main characters of this paper, and show how they naturally give rise to YB-solutions. Motivated by P{\l}onka's $2$-cyclic magmas \cite[Lemma 3]{plonka2}, we first consider in Definitions~\ref{def:rkm} and \ref{def:lkm} the more general concepts of \emph{right} and \emph{left P{\l}onka magmas}, respectively. For example, the variety ${\rm Mag}^{\rm rP}$ of right P{\l}onka magmas consists of magmas $(X, \, \cdot)$ such that for any $x$, $y$, $z\in X$:
\begin{equation*}
  (x \cdot y) \cdot z = (x \cdot z) \cdot y,    \qquad x \cdot (y \cdot z) = x \cdot y.
\end{equation*}
Their structure is given in Theorem~\ref{th:structRPm}: every right P{\l}onka magma is characterized in terms of a partition of the underlying set $X = \coprod_{i\in I} \, X_i$ together with a family of self-maps $f_{i}^{j} : X_{i} \to X_{i}$, for all $i$, $j\in I$, satisfying a certain commutativity condition. We further prove in Theorem \ref{th:uniqpart} that the finest and the coarsest partitions are unique and classify the corresponding structures in Theorem \ref{th:isopart} via the notion of a \emph{connected} family $\cF$ of functions $X\to X$ introduced in Definition \ref{def:connfunc}. The classification of all \emph{ideal-simple} P\l onka magmas is given in Theorems \ref{th:lsimp} and \ref{th:simplestr}: the first result characterizes all objects that are left (= two sided) ideal-simple in ${\rm Mag}^{\rm rP}$ while the second result classifies all those that are right ideal-simple. \\
Definition \ref{def:kbimag} introduces the notion of a {\it P{\l}onka bi-magma}: a set $X$, with two binary operations $\cdot$ and $\ast: X \times X \to X$ such that $(X, \, \cdot)$ is a right P{\l}onka magma, $(X, \, \ast)$ is a left P{\l}onka magma (i.e. the opposite magma $(X, \, \ast^{\rm op})$ is a right P{\l}onka magma) satisfying three compatibility conditions as listed in the aforementioned definition. \\
Building on Theorem \ref{th:structRPm}, the structure of P{\l}onka bi-magmas is established in Theorem~\ref{th:strpbimag} by showing that any P{\l}onka bi-magma can be decomposed as a partition of Lyubashenko's sub-bi-magmas as defined in Example~\ref{exs:bmlyb}(\ref{item:LB}), called a \emph{P\l onka bi-partition} in Definition~\ref{def:biplonka}. The results regarding the uniqueness of the finest and the coarsest decompositions of a P{\l}onka bi-magma as well as their classification are similar to those for right P{\l}onka magmas and are given in Theorems \ref{th:biuniqpart} and \ref{th:biisopart}. Furthermore, Theorem \ref{th:teorema1} and Theorem~\ref{non-deg_lr+rl} prove that every P{\l}onka bi-magma canonically determines a YB-solution, whose unitarity, involutivity and (left-right, resp. right-left) non-degeneracy are characterized by corresponding properties of the underlying (bi-)magmas as follows:

\begin{theoremN}
  Let $(X, \, \cdot, \, \ast)$ be a P{\l}onka bi-magma. Then:

  \begin{enumerate}[(1)]
  \item The associated canonical map defined for any $x$, $y\in X$ by:
    \begin{equation}\label{eq:solgenK0}
      R = R_{(\cdot, \, \ast)} : X\times X \to X\times X, \qquad R (x, \, y) = (x\cdot y, \, x\ast y)
    \end{equation}
    is a YB-solution.

  \item $R$ is unitary if and only if the P{\l}onka bi-magma $(X, \, \cdot, \, \ast)$ is unitary.

  \item $R$ is involutive if and only if $(X, \, \cdot)$ is a right involutory magma and $(X, \, \ast)$ is a left involutory magma.

    \item $R$ is right-left non-degenerate if and only if there exists a P\l onka bi-partition $(X_i,\ f_i^j,\ g_i^j)$ of $(X, \, \cdot, \, \ast)$ in the sense of Definition~\ref{def:biplonka} such that $f_i^j$, $g_i^j : X_i \to X_i$ are bijective for all $i$, $j \in I$. In particular, $R$ is right-left non-degenerate if and only if it is bijective.

  \item $R$ is left-right non-degenerate\footnote{We call \emph{left-right non-degenerate} a solution that is non-degenerate in the sense classically used in the literature, as introduced in \cite[Definition 1.1]{etingof} for the braid equation. See Remark~\ref{re:nede-quasigr} for further details.} if and only if $X = \{ \star \}$ is a singleton set.

  \end{enumerate}
\end{theoremN}
The aforementioned structural results on P{\l}onka bi-magmas are then used to study various types of induced YB-solutions. For instance, a sample of our results concerning ideal-simple YB-solutions can be briefly stated as follows:

\begin{theoremN}
  Let $(X, \, \cdot, \, \ast)$ be a P{\l}onka bi-magma and consider $R = R_{(\cdot, \, \ast)}$ to be the associated canonical YB-solution. Then:

  \begin{enumerate}[(1)]
    \item $R$ is ideal-simple if and only if there exists an incompressible pair $(f,\, g)$ of commuting maps $f$, $g: X\to X$ such that $R = R_{(f, \, g)}$, i.e. $R (x, \, y) = (f(x), \, g(y))$, for all $x$, $y\in X$.
  \item  The set of isomorphism classes of all ideal-simple YB-solutions of this form is parameterized by the set of all triples  $(m,\ n,\ d)\in \bZ_{\ge 0}\times \bZ_{>0}^2$, where $d\le |\bZ/m|$.

  \item\label{sim100intro} If $X$ is a finite set with $|X| = t$, then the number of isomorphism classes of all ideal-simple YB-solutions induced by P{\l}onka bi-magmas is $\sigma (t)$, the sum of all divisors of $t$.

  \end{enumerate}
\end{theoremN}
On the other hand, for the class of bi-connected YB-solutions we prove the following:
\begin{theoremN}
  Let $(X, \, \cdot, \, \ast)$ be a P{\l}onka bi-magma and consider $R = R_{(\cdot, \, \ast)}$ to be the associated canonical YB-solution. Then:

  \begin{enumerate}[(1)]
  \item\label{harary21} $R$ is bi-connected if and only if there exists a connected pair $(f, \, g)$ of two commuting maps
    $f$, $g: X \to X$ such that $R = R_{(f, \, g)}$.
\item In particular, if $\ast = \cdot_{op}$, then $R = R_{(\cdot, \, \cdot_{\rm op})}$ is bi-connected if and only if there exists a connected map $f: X \to X$ such that $R(x, \, y) = (f(x), \, f(y))$, for all $x$, $y\in X$.
Moreover, the number of isomorphism classes of bi-connected solutions of this form associated to right P{\l}onka magmas of order $n$ is
the Harary number $\mathfrak{c} (n)$ \cite[Definition 3.6]{acm1}.
     \end{enumerate}
\end{theoremN}
Furthermore, we show in Proposition~\ref{pr:biconn_indec} that for YB-solutions induced by P{\l}onka bi-magmas bi-connectedness is equivalent to indecomposability, while in Remark~\ref{re:bijunitinv} bijectivity, unitarity and involutivity of such solutions are described in terms of the associated P{\l}onka bi-partitions.

\subsection*{Acknowledgements}
We are very grateful to the referee for carefully reading our manuscript and for their valuable comments and suggestions. We would also like to thank Philip Ture\v{c}ek and Vladimir Antofi for their comments included in Example \ref{exs:exerkm}(\ref{PTVA1}) and Example \ref{exs:exerkm}(\ref{PTVA2}) and C.N. Beli, A. Pilitowska and D. Stanovsk\'{y} for pointers to the literature. Some of the results in this paper were first announced in 2023 in the preprint \cite{CMv1}, which is not intended for publication and will remain available only on arXiv. In preparing the final version, we used ChatGPT (OpenAI) and Claude (Anthropic, Claude Opus~5.5) for proofreading, correcting typos, and making minor improvements.


\section{Preliminaries}\label{se:prel}
For a given set $X$, we denote by $\pi_1$, $\pi_2 \colon X\times X \to X$ the canonical projections $\pi_1 (x, y) = x$, $\pi_2 (x, y) = y$, for all $x$, $y\in X$, and by ${\rm Id}_X$ the identity map on $X$. Furthermore, the group of permutations on $X$ will be denoted by $\Sigma_X$ or by $S_n$ if $X$ is a finite set with $n$ elements. \\
\noindent\textbf{Universal algebras, magmas and bi-magmas.} For the basic concepts of universal algebras we refer to
\cite{adamek_ros, bruck, buris}. We only recall briefly the concepts that will be used throughout the paper. A \emph{universal algebra} (or an \emph{algebra}, for short) is a pair $A = (A, \, (\mu_i)_{i\in I})$,
where $A$ is a set and $\mu_i : A^{n_i} \to A$ is a family of operations on $A$, for all $i\in I$; for any $i\in I$, the integer $n_i \geq 0$ is called the \emph{arity} of $\mu_i$ and $\mu_i$ is called an \emph{$n_i$-ary operation}. Two algebras $A$ and $B$ whose operations are indexed by the same set $I$, with corresponding operations having
the same arity, are said to be of the \emph{same signature (type)}.
A subcategory $\mathcal{V}$ of the category of all universal algebras of a given signature, satisfying a specified family of identities, is called a \emph{variety of algebras}. Basic examples are the variety of all lattices, (semi)groups, rings, the variety of all $R$-modules over a
ring $R$, etc.

All universal algebras used in this paper will have at most two binary operations, with the principal examples being magmas and bi-magmas, which we recall below. A \emph{magma} (\cite[Definition 1.1, Chapter IV]{serre} or \cite[Definition 1]{bourbaki}) is a pair $(X, \, \cdot)$ consisting of a set $X$ with a binary operation $\cdot : X \times X \to X$, $(x, \, y) \mapsto x\cdot y$, called \emph{multiplication}. A magma $(X, \, \cdot)$ is called:
\begin{enumerate}[(1)]
\item \emph{total} if $X = X^2 := \{x\cdot y \, | \, x, y \in X \}$;
\item \emph{idempotent} (or a \emph{band}) if $x^2 = x$, for all $x \in X$;
\item \emph{left} (resp. \emph{right}) \emph{cancellative} if for any $a \in X$ the left (resp. right) translation map $l_a^{\cdot} : X \to X$, $l_a^{\cdot} (x) := a \cdot x$ (resp. $r_a^{\cdot} : X \to X$, $r_a^{\cdot} (x) := x \cdot a$) is injective;
\item \emph{right involutory} (resp. \emph{left involutory}) if all right (resp. left) translations are involutions on $X$, i.e. $r_y^{\cdot} \circ r_{y}^{\cdot} = {\rm Id}_X$ (resp. $l_y^{\cdot} \circ l_{y}^{\cdot} = {\rm Id}_X$) for all $y\in X$. Explicitly, $(X, \, \cdot)$ is right involutory if for any $x$, $y\in X$ we have:
\begin{equation}\label{eq:remaguni}
(x \cdot y) \cdot y = x.
\end{equation}
Note that the identity (\ref{eq:remaguni}) is called the \emph{$2$-cyclic law} in \cite{romanowska}.
\item a \emph{left} (resp. \emph{right}) \emph{quasi-group} \cite{PPZ} if $l_a^{\cdot} : X \to X$ (resp. $r_a^{\cdot} : X \to X$) is bijective, for all $a \in X$;
\item a \emph{quasi-group} if it is simultaneously a left and right quasi-group;
\item a \emph{left zero} (resp. \emph{right zero}) \emph{semigroup} if it has an associative multiplication given by: $x \cdot y : = x$ (resp. $x \cdot y : = y$), for all $x$, $y\in X$;
\item a \emph{cycle set} \cite[Section 4]{rump} if it is a left quasi-group and the following condition holds for any $x$, $y$, $z\in X$
\begin{equation}\label{cycle}
(x \cdot y) \cdot (x \cdot z) = (y \cdot x) \cdot (y \cdot z).
\end{equation}
\end{enumerate}

\begin{examples}
\begin{enumerate}[(1)]
\item The basic example of a non-associative total idempotent magma on the set $\bR$ is the mean operation: $x \cdot y := 2^{-1} (x+y)$, for all $x$, $y\in \bR$;
\item If $G$ has a group structure denoted by juxtaposition and we define $x\cdot y := y x^{-1} y$, for all $x$, $y\in G$, then $(G, \cdot)$ is a right involutory magma.
\end{enumerate}
\end{examples}
A \emph{morphism of magmas} $f: (X, \, \cdot) \to (X', \, \cdot')$ is a function $f: X \to X'$ such that $f(x \cdot y) = f(x) \cdot' f(y)$, for all $x$, $y\in X$. We denote by ${\rm Mag}$ the category of all magmas.\\
A \emph{congruence} \cite[Definition 11, p. 11]{bourbaki} on a magma $(X, \, \cdot)$ is an equivalence relation $\approx$ on the set $X$ compatible with the multiplication: i.e. $x \approx x'$ and $y \approx y'$ imply $x \cdot y \approx x' \cdot y'$. If $\approx$ is a congruence on a magma $(X, \, \cdot)$, then the \emph{quotient magma} is the set $X/\approx$ with the multiplication:
\begin{equation}\label{eq:magfactor}
 \overline{x} \cdot \overline{y} := \overline{x\cdot y}
\end{equation}
for all $\overline{x}$, $\overline{y} \in X/\approx$. The typical example of a congruence is the kernel of a morphism of magmas: if
$f: X \to Y$ is a morphism of magmas, then
\begin{equation*}
{\rm Ker} (f) := \{ (x, \, y) \in X \times X \, | \, f(x) = f(y) \}
\end{equation*}
is a congruence on $X$ and there exists an isomorphism of magmas $X/{\rm Ker} (f)  \cong {\rm Im} (f)$.\\
Recall \cite[page 54]{sim_mag} that a non-empty subset $I \subseteq X$ of a magma $(X, \, \cdot)$ is called a \emph{right} (resp. \emph{left}) \emph{ideal} if $I \cdot X \subseteq I$ (resp. $X \cdot I \subseteq I$) and a \emph{two-sided ideal} if it is both a left and right ideal of $(X, \, \cdot)$. Of course, $X^2 = \{ x\cdot y \, | \, x, \, y\in X\}$ is a two-sided ideal of any
magma $(X, \, \cdot)$.\\
Furthermore, a magma $(X, \cdot)$ is called \emph{right} (resp. \emph{left}) \emph{ideal-simple} if it has no proper non-empty right (resp. left) ideals, and \emph{ideal-simple} if it has no proper non-empty two-sided ideals.

\begin{remark}\label{re:idealsem}
The concept of left/right or two-sided ideal for magmas as defined above, which will be used throughout the paper, is the magma counterpart of
the same term used for semigroups (see, for instance, \cite[p. 26]{Grillet}). Furthermore, the term \emph{ideal}-simple is used to emphasize that the notion refers specifically to the absence of proper non-empty two-sided ideals, distinguishing it from the distinct concept of \emph{congruence}-simple for magmas (see \cite[page 54]{sim_mag}).

It is worth noting that in the theory of universal algebras, the concept of an ideal for algebras of higher signature is considerably more intricate. It was introduced by Ursini and only for universal algebras which have a constant (zero) $0$.
For a detailed discussion we refer to \cite[Definition 1.1]{gumm}.
\end{remark}
A \emph{bi-magma} is a triple $(X, \, \cdot, \, \ast)$ consisting of a
set $X$ with two binary operations $\cdot : X \times X \to X$ and $\ast : X \times X \to X$. Notions of ideal, sub-bi-magma (subalgebra) or quotient bi-magma
are defined in the obvious fashion by extending the ones above associated with a magma: we mention that for bi-magmas the congruences are equivalence relations
on $X$ compatible with both multiplications. In order to distinguish between the various types of ideals, we use left/right adjectives for each of the two operations. For example, a \emph{right-left ideal} of a bi-magma $(X, \, \cdot, \, \ast)$ is a non-empty subset $I \subseteq X$ that is simultaneously a right ideal of $(X, \, \cdot)$ and a left ideal of $(X, \, \ast)$.\\
A \emph{morphism of bi-magmas} $f: (X, \, \cdot, \, \ast) \to (X', \, \cdot', \, \ast')$ is a function $f: X \to X'$ such that $f(x \cdot y) = f(x) \cdot' f(y)$ and
$f(x \ast y) = f(x) \ast' f(y)$, for all $x$, $y\in X$. We denote by ${\rm BiMag}$ the category of all bi-magmas: it is a variety
of universal algebras (empty family of identities) whose objects are sets equipped with two $2$-ary operations.\\
A word of caution regarding terminology is in order: in earlier literature on universal algebras, the terms \emph{groupoid} \cite{bruck, buris} and \emph{$2$-groupoid} \cite{rusii} were used to denote what are now called magmas and bi-magmas, respectively. We avoid this older terminology because the terms groupoid and $2$-groupoid have since acquired different meaning in category theory. For example, nowadays a groupoid typically refers to a category in which every morphism is an isomorphism. Accordingly, we adopt the terminology introduced by Serre \cite[Part I, \S IV.1, Definition 1.1]{serre} and subsequently used by Bourbaki \cite[\S I.1.1, Definition 1]{bourbaki}.\\
Throughout, rather than considering maps $R : X\times X \to X \times X$ we will often work with bi-magmas $(X, \, \cdot, \, \ast)$ in which case we implicitly assume that $R = R_{(\cdot, \, \ast)}$ is given for any $x$, $y\in X$ by:
\begin{equation}\label{notatieR}
R(x, \, y) = R_{(\cdot, \, \ast)} \, (x, \, y) = (x\cdot y, \, x \ast y) = \bigl(l_{x}^{\cdot}(y), \, r_{y}^{\ast}(x) \bigl)
\end{equation}
and refer to it as the \emph{canonical map} associated with the bi-magma $(X, \, \cdot, \, \ast)$. The framework we propose has a number of advantages. On the one hand, it allows us to define, for functions $R: X\times X \to X\times X$, analogues of concepts from universal algebra by translating their bi-magma counterparts into the language of set-theoretic maps $R: X\times X \to X\times X$ (see Section \ref{se:sect2a} for details). On the other hand, for well known concepts that already exist in the Yang-Baxter literature, the language of bi-magmas provides a new perspective for their study and clarifies their relationships with notions that have long been studied in universal algebra. We provide an example below.
\begin{remark}\label{re:nede-quasigr}
The concept of a \emph{non-degenerate} solution was introduced in \cite[Definition 1.1]{etingof} in the context of the braid equation. More precisely, a map $R \colon X \times X \to X\times X$ written as $R(x, \, y) = \bigl(l_{x}^{\cdot}(y), \, r_{y}^{\ast}(x) \bigl) = (x\cdot y, \, x \ast y)$ is called \emph{non-degenerate} if the left translations $l_{z}^{\cdot}$ and the right translations $r_{z}^{\ast}$ are bijections, for all $z\in X$. In our setting, this is equivalent to saying that the bi-magma $(X, \, \cdot, \, \ast)$ associated to $R$ is a left-right quasi-group: i.e. $(X, \, \cdot)$ is a left quasi-group and $(X, \, \ast)$ is a right quasi-group. For this reason, we propose to call such solutions \emph{left-right non-degenerate}. The adjectives \emph{left} and \emph{right} by which we amplify \cite[Definition 1.1]{etingof} are useful bookkeeping tools: a composition by a flip
\begin{equation*}
X\times X\ni (x,y)\xmapsto{\quad\tau_{X}\quad}(y,x)\in X\times X
\end{equation*}
on either side turns \emph{left-right} non-degeneracy into its {\it right-left} counterpart. As solutions $R$ of the braid equation are in bijection
with YB-solutions $\tau_{X}\circ R$, the (left-right) non-degeneracy of \cite[Definition 1.1]{etingof}, a paper focusing on the braid equation, is translated in the literature to the {\it right-left} non-degeneracy of YB-solutions (see \cite{rump}).\\
In fact, several distinct notions of non-degeneracy can be considered. In the equivalent language of bi-magmas $(X, \, \cdot, \, \ast)$ and quasi-groups, these correspond to the left-right, right-left, left-left, right-right, and two-sided cases, where in the latter case both $(X, \, \cdot)$ and $(X, \, \ast)$ are quasi-groups. Apart from these, there are other possible variants when only one of the magmas is a (left, right or two-sided) quasi-group.  It is therefore important to distinguish carefully between these different notions of non-degeneracy and the language of quasi-groups provides an unambiguous framework.\\
As noted above, the existing literature has so far focused almost exclusively on left-right non-degenerate solutions of the braid equation and right-left non-degenerate solutions of the Yang–Baxter equation, leaving the other possibilities largely unexplored. The study of the remaining cases constitutes a natural avenue for future research.
\end{remark}
In addition to (right-left) non-degenerate solutions, several other classes of YB-solutions have received particular attention in the literature. For instance, $R \colon X \times X \to X\times X$ is called
\begin{enumerate}
\item \emph{involutive} if $R^2 = {\rm Id}_{X\times X}$;
\item \emph{diagonal} if $R(x,\, x) = (x,\, x)$ for all $x \in X$;
\item \emph{unitary} if $R^{21}R = {\rm Id}_{X\times X}$, where $R^{21} := \tau_X R \tau_X$.
\end{enumerate}
Let $\mathcal{YB}$ be the category of solutions of the Yang-Baxter equation, i.e. objects of $\mathcal{YB}$ are pairs $(X, \, R)$, where $X$ is a set and $R$ a YB-solution on $X$. A morphism $\sigma: (X, \, R) \to (X', \, R')$ in $\mathcal{YB}$ is a map $\sigma: X \to X'$ such that $(\sigma \times \sigma) \circ R = R' \circ (\sigma \times \sigma)$. This is equivalent to $\sigma: X \to X'$ being a morphism between the associated bi-magmas: i.e. $\sigma (x \cdot y) = \sigma(x) \cdot' \sigma(y)$ and $\sigma (x \ast y) = \sigma(x) \ast' \sigma(y)$, for all $x$, $y\in X$. Two solutions $(X, \, R)$ and $(X', \, R') \in \mathcal{YB}$ are called \emph{isomorphic} if they are isomorphic as objects in the category $\mathcal{YB}$. For a given YB-solution $(X, \, R) \in \mathcal{YB}$ we denote by ${\rm Aut} \, (X, \, R)$ its automorphism group in the category $\mathcal{YB}$.



\begin{definition}\label{def:ybbimagm}
A bi-magma $(X, \, \cdot, \, \ast)$ satisfying the compatibilities \eqref{YB1a}-\eqref{YB3a} will be called a \emph{Yang-Baxter bi-magma} (or a \emph{set-theoretic non-associative Yang-Baxter algebra}) and we denote by ${}_{\rm YB}{\rm BiMag}$ the full subcategory of ${\rm BiMag}$ consisting of all Yang-Baxter bi-magmas.
\end{definition}
The category ${}_{\rm YB}{\rm BiMag}$ is a variety of algebras in the sense of universal algebras and the functor
\begin{equation*}
  {}_{\rm YB}{\rm BiMag} \to \mathcal{YB}, \qquad (X, \, \cdot, \, \ast) \mapsto (X, \, R_{(\cdot, \, \ast)})
\end{equation*}
is of course an isomorphism of categories.

\begin{examples}\label{exs:bmlyb}
  \begin{enumerate}[(1)]
\item\label{item:LB}  The basic example of a YB-solution, attributed to Lyubashenko in \cite[Example 1]{dr}, is the following: $X \neq \emptyset$, for two self-maps $f$, $g: X\to X$ the function defined for any $x$, $y\in X$ by
  \begin{equation}\label{eq:lysol}
    R = R_{(f, \, g)} : X\times X \to X\times X, \qquad R (x, \, y) := \bigl(f(x), \, g(y)\bigl)
\end{equation}
is a solution of the Yang-Baxter equation if and only if $f\circ g = g\circ f$. The map $R_{(f, \, g)}$ given by \eqref{eq:lysol} is called a \emph{Lyubashenko solution} on $X$.  The bi-magma on $X$, associated to $R_{(f, \, g)}$, will be denoted by $X_f^g = (X, \, \cdot_f, \, \ast^g)$, where $x\cdot_f \, y := f(x)$ and $x\ast^g y \, := g(y)$, for all $x$, $y\in X$ and we call it the \emph{Lyubashenko bi-magma} of $f$ and $g$.

\item \label{bbirackb} Let $(X, \, \cdot, \, \ast)$ be a bi-magma. Then $(X, \, \cdot, \, \ast)$ is a Yang-Baxter bi-magma such that $(X, \, \cdot)$ is a right quasi-group and $(X, \, \ast)$ is a left quasi-group if and only if $(X, \, \ast,\, \cdot)$ is a \emph{birack} in the sense of \cite[Lemma 1.2]{Deho}.
\end{enumerate}
\end{examples}




\section{Combinatorial objects: ideal-simple and incompressible maps}\label{se:sect2a}

The isomorphism of categories ${}_{\rm YB}{\rm BiMag} \cong \mathcal{YB}$ established at the end of Section \ref{se:prel} allows us to introduce counterparts of the standard notions from the theory of universal algebras (ideals, (semi)simple, noetherian or artinian objects, and so on) for the category of YB-solutions.
In particular, the notions of an ideal and of a simple object in the category $\mathcal{YB}$ of all YB-solutions are purely combinatorial in nature.
As an application, we explicitly classify all ideal-simple Lyubashenko YB-solutions as given in Example \ref{exs:bmlyb}(\ref{item:LB}):
they are ``almost finite'' and, somewhat surprisingly, are classified by pairs of functions formed by a cycle and a function
which is known in ergodic theory as an \emph{odometer} (or \emph{adding machine map}) \cite[page 266]{np_ergodic}.
To begin with, we introduce:

\begin{definition}\label{def:solsimple}
Let $X$ be a set and $R \colon X \times X \to X\times X$ a function. A non-empty subset $I \subseteq X$ is called an \emph{ideal} of $R$ if
\begin{equation}\label{solsimpleb}
R (I \times X) \subseteq I \times X, \qquad R (X \times I) \subseteq X \times I.
\end{equation}
$R$ is called an \emph{ideal-simple} function if $X$ is the only ideal of $R$. A solution of the Yang-Baxter equation $(X, \, R)\in \mathcal{YB}$ is called an \emph{ideal-simple} YB-solution if $R$ is ideal-simple in the above sense.
\end{definition}

\begin{remarks}\label{re:simplepov}
  \begin{enumerate}[(1)]

  \item\label{item:draci} $I$ is an ideal of $R\colon X \times X \to X\times X$ in the sense of Definition \ref{def:solsimple} if and only if $I$ is a right-left ideal in the associated canonical bi-magma $(X, \, \cdot, \, \ast)$ of $R$: i.e. $I$ is a right ideal of $(X, \, \cdot)$ and a left ideal of $(X, \, \ast)$. In fact, using the flip map, one can define (and we leave it to the reader) all possible versions of ideals for a map $R\colon X \times X \to X\times X$: left-right, left-left, etc. or the one of two-sided ideals.

    In this paper, we are only interested in ideals and ideal-simple YB-solutions in the sense of Definition \ref{def:solsimple}, as they appear to be the natural concepts in the category $\mathcal{YB}$.


  \item If we consider the braid equation, then the flip version of Definition \ref{def:solsimple} has to be considered. More precisely, a non-empty subset $I \subseteq X$ is called a \emph{flip ideal} of $R\colon X \times X \to X\times X$ if
    \begin{equation}\label{solsimpleb2}
      R (I \times X) \subseteq X \times I, \qquad R (X \times I) \subseteq I \times X.
    \end{equation}
    Of course, $I$ is an ideal of $R\colon X \times X \to X\times X$ if and only if $I$ is a flip ideal of $\tau_X \circ R$. 

Another definition of a \emph{simple solution} of the braid equation, following a different viewpoint, appears in \cite[Definition 3.3]{Cedo2}.

  \item\label{noetherian} In the same fashion, we can also introduce various other classes of YB-solutions, which may be of some intrinsic interest. For instance, a YB-solution $(X, \, R)$ can be called \emph{noetherian} (resp. \emph{artinian}) if it satisfies the ACC (resp. DCC) condition on ideals.

  \end{enumerate}
\end{remarks}

\begin{examples}\label{exs:exsolsimple}
  \begin{enumerate}[(1)]
  \item The flip map $\tau_X \colon X \times X \to X\times X$, $\tau_X (x, \, y) = (y, \, x)$ is an ideal-simple YB-solution. In contrast, any non-empty subset of $X$ is an ideal of the identity map ${\rm Id}_{X^2}$.

  \item More generally, every YB-solution of the form $(x,y)\xmapsto{}(g(y), f(x))$ with at least one of the functions $f,g:X\to X$ surjective is ideal-simple.

  \item\label{primideal} Let $f \colon X\to X$ be a function and $R = R_f \colon X \times X \to X\times X$, $R (x, \, y) = (f(x), f(y))$, for all $x$, $y\in X$. Then, $I$ is an ideal of $R_f$ if and only if $I$ is $f$-invariant, i.e. $f(I) \subseteq I$. Thus, $R_f$ is ideal-simple if and only if $X$ is the only $f$-invariant subset of $X$. These maps are classified in Lemma \ref{le:cycles} below.

  \item\label{doiideal} More generally, let $f$, $g: X\to X$ be two commuting self-maps and $R = R_{(f, \, g)} : X\times X \to X\times X$ the associated Lyubashenko YB-solution given by Example \ref{exs:bmlyb}(\ref{item:LB}), i.e. $R (x, \, y) := \bigl(f(x), \, g(y)\bigl)$. Then, $I$ is an ideal of $R_{(f, \, g)}$ if and only if $I$ is both $f$ and $g$ invariant, i.e. $f(I) \subseteq I$ and $g(I) \subseteq I$. It follows that $R_{(f, \, g)}$ is an ideal-simple YB-solution if and only if the pair $(f,\, g)$ is incompressible in the sense of Definition \ref{def:simple} below.
  \end{enumerate}
\end{examples}
Motivated by Example \ref{exs:exsolsimple}(\ref{doiideal}), we introduce the following notions meant to generalize cycles of length $|X|$ in the permutation group $\Sigma_X$:

\begin{definition}\label{def:simple}
Let $X$ be a set and $\cF$ a family of functions $X\to X$. Then:
  \begin{enumerate}[(1)]

  \item\label{item:compres} A subset $I\subseteq X$ {\it compresses $\cF$} if $f(I)\subseteq I$, for all $f\in \cF$.

  \item\label{item:incomprdef} $\cF$ is called {\it incompressible} if no non-empty proper subset of $X$ compresses $\cF$ in the above sense.

   \item\label{item:cycledef} An incompressible singleton $\{f: X \to X \}$ is called an {\it incompressible map} (or an {\it $|X|$-cycle}).
   \end{enumerate}
\end{definition}

The terminology of $|X|$-cycle map is justified by the following:

\begin{lemma}\label{le:cycles}
 Let $X \neq \emptyset$. A function $f: X \to X$ is incompressible precisely when it is isomorphic to the translation by a generator of a finite cyclic group.

  In particular, $R_f: X\times X \to X\times X$, $R_f (x, \, y) = (f(x), f(y))$ is an ideal-simple YB-solution if and only if $X$ is a finite set and $f$ is a cycle of length $|X|$ in the permutation group $\Sigma_X$.
\end{lemma}

\begin{proof}
  That such translations meet the requirement of Definition \ref{def:simple}(\ref{item:cycledef}) is obvious, so we address the converse.

  Note that for every $x\in X$ the corresponding {\it forward orbit}
  \begin{equation}\label{eq:fororb}
    \cO_{x,+}:= \{x,\ f(x),\ f^2(x),\ \cdots\}
  \end{equation}
  satisfies $f(\cO_{x,+})\subseteq \cO_{x,+}$, so that $\cO_{x,+}=X$. This is true of any element of $X$ whatsoever, including $f(x)$. But then $x\in X$ is in the forward orbit of $f(x)$, so that $\cO_{x,+}$ is a finite cycle under $f$.
\end{proof}
It will also be convenient to observe that all members of an incompressible commuting family of functions are necessarily bijective.

\begin{proposition}\label{pr:allarebij}
  If $\cF$ is an incompressible family of mutually commuting functions on a set $X  \neq \emptyset$, every $f\in \cF$ is a bijection $X\to X$. Moreover, the sub-monoid of
  \begin{equation*}
    X^X:=\left\{\text{functions }X\to X\right\}
  \end{equation*}
  generated by $\cF$ is a group.
\end{proposition}
\begin{proof}

  Consider some $x\in X$ and $y := f(x)$. The full forward orbit
  \begin{equation*}
    \cO_{y,+}^{\cF}:=\left\{h(y) \ |\ h\in\text{ monoid $\subseteq X^X$ generated by }\cF\right\}
  \end{equation*}
  under $\cF$ is again $\cF$-invariant, and incompressibility shows that $\cO_{y,+}^{\cF}=X$. In particular, that set contains $x$, so that
  \begin{equation*}
    h(y) = x\text{ for some }h = f_1\cdots f_n,\quad f_i\in \cF.
  \end{equation*}
  But then $F:=ff_1\cdots f_n$ fixes $y$; because $F$ also commutes with every $h\in \cF$, its fixed points constitute an $\cF$-invariant subset of $X$. Once more, incompressibility implies that $F=\id_X$. But because $f$ and the $f_i$ commute, we have
  \begin{equation*}
    f_1\cdots f_n = f^{-1}.
  \end{equation*}
  This proves both claims: an arbitrary $f\in \cF$ is bijective, and its inverse is contained in the monoid generated by $\cF$.
\end{proof}
The preceding discussion affords a classification of sorts for commuting incompressible families. For this purpose, we say that two families $\cF = (f_i)_{i\in I}$ and $\cF' = (g_i)_{i\in I}$ of functions $X\to X$ are
\emph{isomorphic} if there exists a permutation $\sigma \in \Sigma_X$ on $X$ such that $\sigma \circ f_i = g_i \circ \sigma$, for all $i\in I$.

\begin{theorem}\label{th:allcomm}
  \begin{enumerate}[(1)]
  \item\label{item:isgp} The commuting incompressible families
    \begin{equation*}
      \cF=(f_i)_{i\in I}\subseteq X^X
    \end{equation*}
    on a set $X  \neq \emptyset$ consist, up to isomorphism, of the abelian group structures on $X$ together with a choice of an $I$-indexed tuple of elements thereof which generate the group as a monoid.

  \item\label{item:getgp} The group in part (\ref{item:isgp}) can be recovered as the automorphism group
    \begin{equation*}
      \Aut(\cF):=\{\sigma\in \Sigma_X\ |\ \sigma f = f\sigma,\ \forall f\in \cF\}
    \end{equation*}
    of the structure $(X,\cF)$.
  \end{enumerate}
\end{theorem}
\begin{proof}
  \begin{enumerate}[(1)]
  \item Two structures
    \begin{equation*}
      \left(\Gamma,\ (\gamma_i)_{i\in I}\right)
      \quad\text{and}\quad
      \left(\Gamma',\ (\gamma'_i)_{i\in I}\right)
    \end{equation*}
    as in the statement (group structures on $X$ together with tuples of generators) are considered isomorphic under the obvious circumstances: there is a group isomorphism $\Gamma\cong \Gamma'$ sending $\gamma_i$ respectively onto $\gamma_i'$. The statement is little more than a reformulation of Proposition \ref{pr:allarebij}, but we elaborate.

    We already know from Proposition \ref{pr:allarebij} that the monoid generated by the $f_i$ is a group $\Gamma$. Note that $\Gamma$ acts freely on $X$: the set of fixed points of any $\gamma\in \Gamma$ compresses $\cF$ in the sense of Definition \ref{def:simple}(\ref{item:compres}) so as soon as $\gamma$ has any fixed points at all it must be trivial.

    $\Gamma$ also acts transitively on $X$: every $\Gamma$-orbit compresses $\cF$, so $X$ must constitute a single orbit.

    We can thus identify $X\cong \Gamma$ (so that the action becomes translation) by choosing an element $x_0\in X$ to map to the identity $1\in\Gamma$. The different choices give isomorphic sets with $\cF$-action, and once we have fixed $x_0=1$ all other isomorphisms respect the group structure because they commute with translation by the (monoid) generators $f_i$ of $\Gamma$.

  \item For the last part, observe that an automorphism of $(X,\cF)$ is nothing but a bijection $X\to X$ commuting with the members of $\cF$; since those members generate $X$ as a group (even monoid), those bijections are precisely the translations by elements $x\in X$.
  \end{enumerate}
\end{proof}
Counting the isomorphism classes of finite commuting incompressible families on a given finite set $X$ is now a fairly simple matter:

\begin{corollary}\label{cor:howmany}
Let $X$ be a finite set with $|X| = t$ and $k\ge 1$. The number of isomorphism classes
of commuting incompressible $k$-tuples $\cF =(f_1, \dots, f_k)$ of maps $X\to X$ is
  \begin{equation}\label{eq:divchain}
    \sum_{1\le d_1|d_2|\cdots|d_{k-1}|t} d_1\cdot d_2\cdot\ldots\cdot d_{k-1}.
  \end{equation}
\end{corollary}
\begin{proof}
  The proof proceeds by induction on $k$. For the base case $k=1$ there is only one chain of divisors as in \eqref{eq:divchain}, namely the empty chain (since we are counting {\it up} from $1$ to $k-1=0$). But then the corresponding empty-product term in \eqref{eq:divchain} is 1 (as all empty numerical products are). This equates \eqref{eq:divchain} to 1, hence the conclusion: by Lemma \ref{le:cycles}, the only incompressible single function on $X$ is the length-$t$ cycle.

  Assume, next, that the claim has been proven for $|\cF| \le k$ (for some $k\ge 1$ fixed throughout the rest of the proof). Following \cite[equation (2), p.45]{zbMATH07852702} we write $B(k,t)$ for \eqref{eq:divchain}, so that $B(1,-)\equiv 1$. The higher $B(\ell,-)$ are then computable via the recursion
  \begin{equation*}
    B(k+1,t)
    =
    \sum_{d|t}d\cdot B(k,d).
  \end{equation*}
  Given the induction base case, it thus remains to argue that the numbers $N(k,t)$ of isomorphism classes of commuting incompressible $k$-tuples on $X$ satisfy the same recurrence:
  \begin{equation*}
    N(k+1,t)
    =
    \sum_{d|t}d\cdot N(k,d).
  \end{equation*}
  This follows fairly easily from Theorem \ref{th:allcomm}: specifying a group with $(k+1)$ generators $\left(\Gamma,(\gamma_i)_{i=1}^{k+1}\right)$ on $X$ entails
  \begin{itemize}
  \item first specifying the subgroup
    \begin{equation*}
      \Gamma_{small}:=\braket{\gamma_1,\ \cdots,\ \gamma_k}
    \end{equation*}
    generated by the first $k$ of those generators, meaning a divisor $d$ of $t$ (the order of that smaller group) and one of the $N(k,d)$ structures on that subgroup;

  \item and then, having made the identification
    \begin{equation*}
      \Gamma/\Gamma_{small}\cong \bZ/(t/d)
      \quad\text{so that the image of $\gamma_{k+1}$ is 1},
    \end{equation*}
    specifying which of the $d$ elements of $\Gamma_{small}$ equals $\gamma_{k+1}^{t/d}$.
  \end{itemize}
\end{proof}
The present discussion also allows for a classification of ideal-simple Lyubashenko YB-solutions. According to Example \ref{exs:exsolsimple}(\ref{doiideal}), we have to classify all incompressible pairs $(f, \, g)$ of two commuting self-maps $f$, $g: X \to X$. This will be a straightforward consequence of Theorem \ref{th:allcomm}, but we phrase the result somewhat more concretely than in that statement.

\begin{corollary}\label{cor:simpbimag}
  The ideal-simple Lyubashenko YB-solutions of Example \ref{exs:bmlyb}(\ref{item:LB}) are precisely those with the data $(X, \ f,\ g)$ classified up to isomorphism by triples
  \begin{equation*}
    (m,\ n,\ d)\in \bZ_{\ge 0}\times \bZ_{>0}^2,\quad d\le |\bZ/m|
  \end{equation*}
  as follows:
  \begin{itemize}
  \item $X=\bZ/m\times \bZ/n$ (with $\bZ/m=\bZ$ when $m=0$);
  \item $f$ is translation by $(1,0)$;
  \item and
    \begin{equation*}
      \bZ/m\times \bZ/n\ni
      (a,b)
      \xmapsto{\quad g\quad}
      \begin{cases}
        (a,b+1) &\text{if}\quad 0\le b<n-1\\
        (a-d,0) &\text{if}\quad b=n-1.\\
      \end{cases}
    \end{equation*}
  \end{itemize}
\end{corollary}
\begin{proof}
  By Theorem \ref{th:allcomm}, commuting incompressible pairs $(f,g)$ on $X$ can be identified (up to isomorphism) with pairs of monoid generators for an abelian group structure on $X$. One first specifies the cyclic group $\Gamma\cong \bZ/m$ (possibly infinite, corresponding to $m=0$) generated (as a group) by $f$, with the latter acting as translation by $1\in \bZ/m$. $g$ operates incompressibly on the quotient $X/\Gamma$, so that quotient must be a finite cyclic group $\bZ/n$. Finally, because $f$ and $g$ generate $X$ as a {\it monoid} (not just group), $g^n\in\Gamma$ must be translation by some {\it negative} $-d$.
\end{proof}

\section{Solutions to the Yang-Baxter equation via P{\l}onka (bi)-magmas}\label{se:plonkabi}

In this section we introduce the main characters of our paper, namely \emph{P{\l}onka bi-magmas}. Building on ideas from previous sections, we will arrive naturally at this concept and the induced YB-solution.
For the reader's convenience, we describe briefly the organization of this section. In Subsection~\ref{sse:plonka} we introduce right and left P{\l}onka magmas as well as P{\l}onka bi-magmas. Subsection~\ref{sse:ybsol} contains the first main result of the section, Theorem~\ref{th:teorema1}, which states that every P{\l}onka bi-magma induces a YB-solution and characterizes its unitarity and involutivity. Subsection~\ref{sse:strright} provides the structure theorem for right P{\l}onka magmas in Theorem~\ref{th:structRPm} together with the uniqueness and classification results for their coarsest and finest P{\l}onka partitions in Theorems~\ref{th:uniqpart} and \ref{th:isopart}. The aforementioned results are applied in Subsection~\ref{sse:idsimpright} to describe the (left, right) ideal-simple right P{\l}onka magmas in Theorems~\ref{th:lsimp} and \ref{th:simplestr}. Subsection~\ref{sse:strbi} is devoted to studying the structure of P{\l}onka bi-magmas in Theorem~\ref{th:strpbimag} together with its consequences for non-degenerate and bijective YB-solutions in Theorem~\ref{non-deg_lr+rl}, Corollary~\ref{thm:class_R_L_Non} and Remark~\ref{re:bijunitinv}. Finally, in Subsection~\ref{sse:idsimpbic} we classify the ideal-simple and the bi-connected YB-solutions induced by P{\l}onka bi-magmas in Theorems~\ref{cor:simpbimagb} and \ref{th:teorema12} and Proposition~\ref{pr:biconn_indec}.

\subsection{P{\l}onka magmas and P{\l}onka bi-magmas}\label{sse:plonka}
We begin by recalling P{\l}onka's concept of a \emph{$2$-cyclic magma}, as first introduced in \cite[Lemma 3]{plonka2}.
\begin{definition}\label{def:2cyc}
A magma $(X, \, \cdot)$ is called a \emph{$2$-cyclic magma} if the multiplication satisfies the following compatibility conditions for any $x$, $y$, $z\in X$:
\begin{equation}\label{eq:ecyccom}
(x \cdot y) \cdot z = (x \cdot z) \cdot y,   \quad x \cdot (y \cdot z) = x \cdot y, \quad x \cdot x = x, \quad (x \cdot y) \cdot y = x.
\end{equation}
\end{definition}
The first axiom of \eqref{eq:ecyccom} is just a commutativity condition for all right translation maps and the third axiom is an idempotent/band condition for the magma $(X, \, \cdot)$. The last compatibility of \eqref{eq:ecyccom} is the right involutory law as given by \eqref{eq:remaguni}: all right translation maps are involutions on $X$. The second compatibility condition of \eqref{eq:ecyccom} is called the \emph{right reduction law} in \cite{romanowska}.

\begin{remarks}\label{res:plwork}
  \begin{enumerate}[(1)]
  \item Using the right translation maps (i.e. $r_{z}^{\cdot}(x) = x\cdot z$ for all $x$, $z \in X$) the four compatibility conditions of \eqref{eq:ecyccom} can be rewritten as follows:
    \begin{equation}\label{eq:ecyccomb}
      r_{z}^{\cdot} \circ r_{y}^{\cdot} = r_{y}^{\cdot} \circ r_{z}^{\cdot}, \quad r_{r_{z}^{\cdot} (y)}^{\cdot} = r_{y}^{\cdot}, \quad  r_{y}^{\cdot}(y) = y, \quad  r_{y}^{\cdot} \circ r_{y}^{\cdot}  = {\rm Id}_X
    \end{equation}
    for all $y$, $z \in X$. Thus, defining a $2$-cyclic magma $(X, \, \cdot)$ is equivalent to specifying a pair $(X, \, (r_{x})_{x\in X})$, consisting of a family of self-maps $r_{x} : X \to X$, indexed by the set $X$, satisfying the compatibility conditions \eqref{eq:ecyccomb}.
  \item The class of $2$-cyclic magmas is one of the four types of varieties of algebras in P{\l}onka's structure theorem \cite[Theorem 1]{plonka2} on universal algebras having exactly $n$ $n$-ary operations depending on all their variables for each positive integer $n \geq 1$.
  \item\label{quan2cyl} We mention that the concept of a \emph{$2$-reductive involutory medial quandle} as introduced and studied in \cite{ppsz} is equivalent to P{\l}onka's $2$-cyclic magma. More precisely,  $(X, \, \cdot)$ is a $2$-reductive involutory medial quandle in the sense of \cite{ppsz} if and only if the opposite magma $(X, \cdot_{\rm op})$ is a P{\l}onka $2$-cyclic magma.
  \end{enumerate}
\end{remarks}
We now consider the following weakening of Definition \ref{def:2cyc}, in which only the first two compatibility conditions are kept:

\begin{definition}\label{def:rkm}
  A magma $(X, \, \cdot)$ is called a \emph{right P{\l}onka magma} if the multiplication satisfies the following two compatibility conditions for any $x$, $y$, $z\in X$:
\begin{equation}\label{eq:rkm1}
(x \cdot y) \cdot z = (x \cdot z) \cdot y,    \qquad x \cdot (y \cdot z) = x \cdot y.
\end{equation}
We denote by ${\rm Mag}^{\rm rP}$ the full subcategory of ${\rm Mag}$ of all right P{\l}onka magmas.
\end{definition}

\begin{remarks}\label{motivplon}
\begin{enumerate}[(1)]
\item\label{defechiv}  The adjective \emph{right} introduced in  Definition \ref{def:rkm} comes from the fact that the two axioms of \eqref{eq:rkm1} can be rewritten equivalently using the right translation maps. In fact, defining a right P{\l}onka magma is equivalent to giving a tuple $\bigl(X, \, (r_x)_{x\in X} \bigl)$, consisting of a set $X$ and a family of self-maps $r_x : X \to X$, indexed by the set $X$, such that for 
    any $x$, $y \in X$:
\begin{equation}\label{ecyccomccc}
      r_x \circ r_y = r_y \circ r_x, \quad r_{r_x (y)} = r_y.
\end{equation}
 Two pairs $(X, \, (r_x)_{x\in X})$ and $(Y, \, (s_y)_{y\in Y})$ as above are called \emph{isomorphic} if the corresponding right P{\l}onka magmas are isomorphic as usual magmas. We can easily prove that this is equivalent to the fact that there exists a bijective map $\sigma : X \to  Y$ such that $s_{\sigma (x)} = \sigma \circ r_x \circ \sigma^{-1}$, for all $x\in X$. In particular, we obtain that the automorphism group of the right P{\l}onka magma $(X, \, (r_x)_{x\in X})$ is equal to:
    \begin{equation}\label{eq:grauto}
      {\rm Aut} \, \bigl(X, \, (r_x)_{x\in X} \bigl) = \{\sigma \in \Sigma_X \, | \, \sigma \circ r_x  = r_{\sigma (x)} \circ \sigma, \,  \forall x\in X\}\leq \Sigma_X
    \end{equation}
    which is a twisted version of what we have called in \cite[Definition 1.1]{acm1} the \emph{centralizer} or \emph{automorphism group} of the tuple $(r_x)_{x\in X}$ of self-maps of $X$.
\item A $2$-cyclic magma as defined in Definition \ref{def:2cyc} is exactly a right involutory idempotent P{\l}onka magma $(X, \, \cdot)$. We mention that right band P{\l}onka magmas were studied in \cite{romanowska} under the name of \emph{LIR-groupoids}.
\end{enumerate}
\end{remarks}

\begin{examples}\label{exs:exerkm}
  \begin{enumerate}[(1)]
\item Any sub-magma of a right P{\l}onka magma $(X, \cdot)$ and any quotient $(X/ \approx, \, \bullet)$, in the sense of \eqref{eq:magfactor} for a congruence on $X$, is again
  a right P{\l}onka magma. The left-zero semigroup structure on any set $X$ is a right P{\l}onka magma.

   \item\label{generic} If $f: X \to X$ is a function, then $(X, \, \cdot_f)$ is a right P{\l}onka magma, where the multiplication $\cdot_f$ is given by
    $x \cdot_f y := f(x)$, for all $x$, $y\in X$. Moreover, $(X, \, \cdot_f)$ is a right involutory magma if and only if $f$ is an involution on $X$, i.e.
    $f^2 = {\rm Id}_X$.

    Two magmas $(X, \, \cdot_f)$ and $(X, \, \cdot_g)$ are isomorphic if and only if the functions $f$, $g: X\to X$ are conjugate, i.e. there exists
    $\sigma \in \Sigma_X$ such that $g = \sigma \circ f \circ \sigma^{-1}$. Thus, if $X$ is a finite set with $|X| = n$, then the number
    of isomorphism classes of all right P{\l}onka magmas of the form $(X, \, \cdot_f)$ is the Davis number $d(n)$ (\cite[Remark 1.4]{acm1}). On the other hand, the number
    of isomorphism classes of all right involutory magmas of the form $(X, \, \cdot_f)$ is  $1+\left\lfloor\frac n2\right\rfloor$ (see the proof of \cite[Corollary 2.24]{acm1}).

    \item\label{PTVA1} The following observation was kindly communicated to us by Philip Ture\v{c}ek: a right P{\l}onka magma $(X, \, \cdot)$ is associative
    if and only if there exists an idempotent self-map $f = f^2: X \to X$ such that $(X, \, \cdot) = (X, \, \cdot_f)$. In particular, using say \cite[Lemma 1.4]{acm1}, we obtain that the number of isomorphism classes of all associative right P{\l}onka magmas on a set with $n$ elements is the Euler partition number $p(n)$.

    Indeed, assume first that $(X, \, \cdot)$ is an associative right P{\l}onka magma; then we have:
    $x y = xyz = xzy = xz$, for all $x$, $y$, $z\in X$; in particular, we obtain that $x y = x^2$, for all $x$, $y\in X$, i.e. each line in the Cayley table is the constant $x^2$. Now, thanks to the reduction law, the map $f: X \to X$, $f (x) := x ^2$ is an idempotent map and thus $(X, \, \cdot) = (X, \, \cdot_f)$; the converse
    is straightforward.

   \item If a right P{\l}onka magma $(X, \, \cdot)$ is commutative, then $(X, \, \cdot)$ is the constant magma, i.e. there exists $c\in X$ such that $x\cdot y := c$, for all $x$, $y\in X$. Indeed, let $(X, \, \cdot)$ be a commutative right P{\l}onka magma. Then, we have:
    $(x \cdot y) \cdot z = (y \cdot x) \cdot z = (y \cdot z) \cdot x = x \cdot (y \cdot z)$, for all $x$, $y$, $z\in X$. Thus, $(X, \, \cdot)$ is associative and hence there exists an idempotent map $f: X \to X$ such that $x \cdot y = f(x)$, for all $x$, $y\in X$. But then $f (x) = x \cdot y = y\cdot x = f(y)$, i.e. $f$ is the constant map.

    \item\label{PTVA2} The  number of isomorphism classes of all right involutory P{\l}onka magmas on a set with $1, \cdots, 10$ elements is: $1$, $2$, $4$, $12$, $37$, $164$, $849$, $6081$, $56164$, $698921$. It is the sequence \cite{slo-A361720} that was kindly computed at our request by Philip Ture\v{c}ek and Vladimir Antofi using a computer program. On the other hand, the number of isomorphism classes of all right P{\l}onka magmas has a much more rapid growth: on a set with $1$, $2$ or $3$ elements this number is: $1$, $3$, $11$.
\end{enumerate}
\end{examples}
The left hand version of Definition \ref{def:rkm} is the following:

\begin{definition}\label{def:lkm}
A magma $(X, \, \ast)$ is called a \emph{left P{\l}onka magma} if the multiplication satisfies the following compatibility conditions for any $x$, $y$, $z\in X$:
\begin{equation}\label{eq:remag}
x \ast (y \ast z) = y \ast (x \ast z),  \qquad (x \ast y) \ast z = y \ast z.
\end{equation}
The full subcategory of ${\rm Mag}$ of all left P{\l}onka magmas will be denoted by ${}^{\rm lP}{\rm Mag}$.
\end{definition}

\begin{remark}
It is worth pointing out that any left P{\l}onka magma $(X, \, \ast)$ which is also a left quasi-group is a cycle set as defined in \cite[Section 4]{rump}. Indeed, using (\ref{eq:remag}) we have for any $x$, $y$, $z\in X$
$$
(x \ast y) \ast (x \ast z) = y \ast (x \ast z) = x \ast (y \ast z) = (y \ast x) \ast (y \ast z).
$$
Therefore (\ref{cycle}) holds, proving that the left quasi-group $(X, \, \ast)$ is a cycle set.
\end{remark}
Right and left P{\l}onka magmas are now used to introduce a new variety of universal algebras which
will be shown to be a full subcategory of ${}_{\rm YB}{\rm BiMag}$ of all Yang-Baxter bi-magmas as defined in Definition \ref{def:ybbimagm}.

\begin{definition}\label{def:kbimag}
A \emph{P{\l}onka bi-magma} is a bi-magma $(X, \, \cdot, \, \ast)$ such that:

(PM1) $(X, \, \cdot)$ is a right P{\l}onka magma;

(PM2) $(X, \, \ast)$ is a left P{\l}onka magma;

(PM3) The following compatibility conditions hold for any $x$, $y$, $z\in X$:
\begin{align}
  x \ast (y \cdot z) &= (x \ast y) \cdot z\label{eq:KM3a}\\
  (x \cdot z) \ast y  &= x \ast y\label{eq:KM3b} \\
  x \cdot (y\ast z) &= x \cdot z.\label{eq:KM3c}
\end{align}
A P{\l}onka bi-magma $(X, \, \cdot, \, \ast)$ is called \emph{unitary} if for any $x$, $y\in X$:
\begin{equation}\label{eq:uniKbimag}
(x \ast y) \cdot x = y.
\end{equation}
We denote by ${\rm PBiMag}$ the full subcategory of ${\rm BiMag}$ of all P{\l}onka bi-magmas which is a  variety of algebras in the sense of universal
algebras.
\end{definition}

\begin{remark}\label{pp2red}
Four of the seven axioms that define a P{\l}onka bi-magma in Definition \ref{def:kbimag} also appear in a different context in the work of P. Jedli\v{c}ka and A. Pilitowska \cite[Definition 3.1]{PP1}. Indeed, let $(X, \, \cdot, \, \ast)$ be a bi-magma and define, for any $x$, $y\in X$ the family of maps $\sigma_x (y) := x \ast y$ and $\tau_y (x) := x\cdot y$. Then, we can prove that the triple $(X, \, \sigma, \, \tau)$ is \emph{$2$-reductive} in the sense of \cite[Definition 3.1]{PP1} if and only if the compatibility conditions \eqref{eq:KM3b}, \eqref{eq:KM3c} and the second axioms of \eqref{eq:rkm1} and \eqref{eq:remag} hold.
\end{remark}

\begin{examples}\label{exs:exkimbimag}
  \begin{enumerate}[(1)]

  \item\label{triviKim} Any set $X$ has a \emph{trivial} unitary P{\l}onka bi-magma structure via the left/right-zero semigroup structures: $x\cdot y := x$ and $x \ast y := y$, for all $x$, $y\in X$. Moreover, if $(X, \, \ast)$ is the right-zero semigroup, then $(X, \, \cdot, \, \ast)$ is a P{\l}onka bi-magma if and only if $(X, \, \cdot)$ is a right P{\l}onka magma. Dually, if $(X, \, \cdot)$ is the left-zero semigroup, then $(X, \, \cdot, \, \ast)$ is a P{\l}onka bi-magma if and only if $(X, \, \ast)$ is a left P{\l}onka magma.

  \item\label{lefbim} Let $(X, \, \cdot)$ be a right P{\l}onka magma. Then we can easily prove that $(X, \, \cdot, \, \ast := \cdot_{\rm op})$ is a P{\l}onka bi-magma, where
    $\cdot_{\rm op}$ is the opposite magma of $(X, \, \cdot)$. Furthermore, $(X, \, \cdot, \, \cdot_{\rm op})$ is a unitary bi-magma if and only if $(X, \, \cdot)$ is a right involutory magma. In the same fashion, if $(X, \, \ast)$ is a left P{\l}onka magma, then $(X, \, \cdot :=\ast_{\rm op}, \, \ast)$ is a P{\l}onka bi-magma (that
    is unitary if and only if $(X, \, \ast)$ is a left involutory magma).

  \item\label{LyPl_can} Let $f$, $g : X \to X$ be two self-functions and define $x\cdot_f y := f(x)$ and $x\ast^g y :=g(y)$, for all $x$, $y\in X$. Then we can see that
    $X_f^g = (X, \, \cdot_f, \, \ast^g)$ is a P{\l}onka bi-magma if and only if $f \circ g = g \circ f$, i.e.  $X_f^g = (X, \, \cdot_f, \, \ast^g)$ is the
    Lyubashenko bi-magma from Example \ref{exs:bmlyb}(\ref{item:LB}). Furthermore, $(X, \, \cdot_f, \, \ast^g)$ is a unitary bi-magma if and only if $f$ and $g$ are bijections and $g = f^{-1}$.
\end{enumerate}
\end{examples}
\subsection{YB-solutions induced by P{\l}onka bi-magmas}\label{sse:ybsol}
We can now construct new families of YB-solutions arising from P{\l}onka bi-magmas:

\begin{theorem}\label{th:teorema1}
  Let $(X, \, \cdot, \, \ast)$ be a P{\l}onka bi-magma. Then:

  \begin{enumerate}[(1)]
  \item The canonical associated map defined for any $x$, $y\in X$ by:
    \begin{equation}\label{eq:solgenK}
      R = R_{(\cdot, \, \ast)} : X\times X \to X\times X, \qquad R (x, \, y) = (x\cdot y, \, x\ast y)
    \end{equation}
    is a solution of the set-theoretic Yang-Baxter equation.

  \item $R$ is unitary if and only if the P{\l}onka bi-magma $(X, \, \cdot, \, \ast)$ is unitary.

  \item $R$ is involutive if and only if $(X, \, \cdot)$ is a right involutory magma and $(X, \, \ast)$ is a left involutory magma.
  \end{enumerate}
\end{theorem}

\begin{proof}
  \begin{enumerate}[(1)]
  \item In the proof we will use all seven axioms from the definition of the concept of P{\l}onka bi-magma. Let $x$, $y$, $z\in X$. In what follows, \eqref{eq:rkm1}$_1$ and \eqref{eq:rkm1}$_2$ denote the first and the second identity in \eqref{eq:rkm1}, respectively, and similarly for \eqref{eq:remag}. Using \eqref{eq:KM3c}, \eqref{eq:remag}, \eqref{eq:rkm1} and \eqref{eq:KM3b}, we have:


    \begin{align*}
      R^{12} R^{13} R^{23} (x, \, y, \, z)  & =   R^{12} R^{13} (x, \, y\cdot z, \, y\ast z ) && \text{(definition of $R^{23}$)} \\
                                            & =  R^{12} \bigl( x\cdot (y\ast z), \, y\cdot z, \, x \ast (y \ast z)   \bigl)   && \text{(definition of $R^{13}$)} \\
                                            & =  R^{12} \bigl( x\cdot z, \, y\cdot z, \, y \ast (x \ast z)   \bigl)  && \text{(by \eqref{eq:KM3c} and \eqref{eq:remag}$_1$)} \\
                                            & =  \bigl( (x\cdot z) \cdot (y \cdot z), \,  (x\cdot z) \ast (y\cdot z), \, y \ast (x \ast z)   \bigl) && \text{(definition of $R^{12}$)} \\
                                            & =  \bigl( (x\cdot z) \cdot y, \,  x \ast (y\cdot z), \, y \ast (x \ast z)   \bigl) && \text{(by \eqref{eq:rkm1}$_2$ and \eqref{eq:KM3b})}
    \end{align*}
    and using \eqref{eq:rkm1}, \eqref{eq:KM3b}, \eqref{eq:KM3a}, \eqref{eq:remag} and \eqref{eq:KM3c}, we have:
    \begin{align*}
      R^{23} R^{13} R^{12} (x, \, y, \, z)  & =   R^{23} R^{13} (x \cdot y, \, x\ast y, \, z ) && \text{(definition of $R^{12}$)} \\
                                            & =  R^{23} \bigl( (x\cdot y) \cdot z, \, x\ast y, \, (x \cdot y) \ast z   \bigl)   && \text{(definition of $R^{13}$)} \\
                                            & =  R^{23} \bigl( (x\cdot z) \cdot y, \, x\ast y, \, x\ast z  \bigl)  && \text{(by \eqref{eq:rkm1}$_1$ and \eqref{eq:KM3b})} \\
                                            & =  \bigl( (x\cdot z) \cdot y, \, (x\ast y) \cdot (x\ast z), \, (x\ast y) \ast (x \ast z)   \bigl) && \text{(definition of $R^{23}$)} \\
                                            & =  \bigl( (x\cdot z) \cdot y, \,  x \ast \bigl(y \cdot (x \ast z) \bigl), \, y \ast (x \ast z)   \bigl) && \text{(by \eqref{eq:KM3a} and \eqref{eq:remag}$_2$)} \\
                                            & =  \bigl( (x\cdot z) \cdot y, \,  x \ast (y \cdot z), \, y \ast (x \ast z)   \bigl) && \text{(by \eqref{eq:KM3c})}
    \end{align*}
    as needed.

  \item First we observe that
    $R^{21} (x, \, y) = (\tau_X \circ R \circ \tau_X ) (x, \, y) = (y\ast x, \, y\cdot x)$. Thus, using \eqref{eq:remag}, \eqref{eq:rkm1} and
    \eqref{eq:KM3a} we have:
    \begin{align*}
      R^{21} R (x, \, y)  & =   R^{21} (x\cdot y, \, x\ast y ) && \text{(definition of $R$)} \\
                          & =  \bigl( (x \ast y) \ast (x \cdot y), \, (x \ast y) \cdot (x \cdot y) \bigl)   && \text{(definition of $R^{21}$)} \\
                          & =  \bigl( y \ast (x \cdot y), \, (x \ast y) \cdot x \bigl)  && \text{(by \eqref{eq:remag}$_2$ and \eqref{eq:rkm1}$_2$)} \\
                          & =  \bigl( (y\ast x) \cdot y, \, (x \ast y) \cdot x \bigl) && \text{(by \eqref{eq:KM3a})}
    \end{align*}
    for all $x$, $y\in X$. Thus, $R$ is a unitary solution if and only if \eqref{eq:uniKbimag} holds, i.e. the P{\l}onka bi-magma $(X, \, \cdot, \, \ast)$ is unitary.

  \item Taking into account \eqref{eq:KM3b} and \eqref{eq:KM3c} we easily obtain that $R^2 (x, \, y) = \Bigl( (x \cdot y) \cdot y, \, x \ast (x \ast y) \Bigl)$.  Thus, $R^2 = {\rm Id}_{X\times X}$ if and only if $(X, \, \cdot)$ is a right involutory magma and $(X, \, \ast)$ is a left involutory magma.
         \end{enumerate}
\end{proof}

\begin{remarks}\label{nedekim}
  \begin{enumerate}[(1)]
  \item Let $(X, \, \cdot, \, \ast)$ and $(X', \, \cdot', \, \ast')$ be two P{\l}onka bi-magmas. Then the associated YB-solutions
    $(X, \, R_{(\cdot, \, \ast)})$ and $(X', \, R_{(\cdot', \, \ast')})$ as given by Theorem \ref{th:teorema1} are isomorphic in $\mathcal{YB}$ if and only if
    $(X, \, \cdot, \, \ast) \cong (X', \, \cdot', \, \ast')$ as P{\l}onka bi-magmas. Furthermore,
    ${\rm Aut} (X, \, R_{(\cdot, \, \ast)}) = {\rm Aut} (X, \, \cdot) \cap {\rm Aut} (X, \, \ast) \leq \Sigma_X$. We note that the last two automorphism
    groups are twisted centralizers of all right/left translations as defined in \eqref{eq:grauto}.

  \item Let $(X, \, \cdot, \, \ast)$ be a P{\l}onka bi-magma. Then, as noted in Theorem~\ref{non-deg_lr+rl} below, the YB-solutions $R = R_{(\cdot, \, \ast)} : X\times X \to X\times X$ constructed in \eqref{eq:solgenK}
  are never left-right non-degenerate (unless $|X| = 1$), but they could be bijective, involutive, unitary or (right-left) non-degenerate (see Remark~\ref{re:bijunitinv} below for the precise relationship between these properties). For instance, if $R_{(\cdot, \, \ast)}$ is involutive, then it is (right-left) non-degenerate. Indeed, if $R_{(\cdot, \, \ast)}$ is involutive, then $(X, \, \cdot)$ is a right involutory (resp. $(X, \, \ast)$ is a left involutory) magma. Thus, all right translation (resp. left translation) maps are involutions on $X$, in particular bijective. This observation provides further evidence for Remark \ref{re:nede-quasigr} by highlighting another member of the broader non-degeneracy family.
  \end{enumerate}
\end{remarks}
As a special case of Theorem \ref{th:teorema1} we obtain another canonical embedding

\begin{equation*}
{\rm Mag}^{\rm rP} \lhook\joinrel\xrightarrow{\quad}  {}_{\rm YB}{\rm BiMag} \cong \mathcal{YB}, \qquad (X, \, \cdot) \mapsto (X, \, \cdot, \, \cdot_{\rm op})
\end{equation*}
Equivalently rephrased, we have:

\begin{corollary}\label{cor:famI}
Let $(X, \, \cdot)$ be a right P{\l}onka magma. Then:
 \begin{enumerate}[(1)]
\item The map
\begin{equation}\label{eq:famII}
R = R_{(\cdot, \, \cdot_{\rm op})} : X\times X \to X\times X, \qquad R (x, \, y) = (x\cdot y, \, y\cdot x)
\end{equation}
is a solution of the Yang-Baxter equation.

\item $R$ is unitary if and only if $(X, \, \cdot)$ is a right involutory magma;

\item $R$ is unitary and diagonal if and only if $(X, \, \cdot)$ is a $2$-cyclic magma.
\end{enumerate}
\end{corollary}

\begin{proof}
  Indeed, if $(X, \, \cdot)$ is a right P{\l}onka magma, then using Example \ref{exs:exkimbimag}(\ref{lefbim}), we obtain that $(X, \, \cdot, \, \ast := \cdot_{\rm op})$ is a P{\l}onka bi-magma. The proof of (1)-(3) now follows from Theorem \ref{th:teorema1}. We only mention that in this case $R^{21} = R$; hence, $R$ is a unitary solution if and only if $R$ is involutive.
  \end{proof}
\subsection{The structure of right P{\l}onka magmas}\label{sse:strright}
The next result can be seen as the right P{\l}onka magma analogue of the structure theorem for Kimura semigroups \cite[Theorem 2]{Kim2}. It rephrases the concept of right P{\l}onka magmas in an equivalent way using the combinatorial language of partitions on sets and self-maps rather than in terms of universal algebras. The proof follows the approach used by P{\l}onka for another variety of algebras \cite[Theorem 3]{plonka2}.

\begin{theorem}\label{th:structRPm}
  Let $X$ be a set. Then:

  \begin{enumerate}[(1)]

  \item\label{item:RPma} $(X, \, \cdot)$ is a right P{\l}onka magma if and only if there exists a partition $X = \coprod_{i\in I} \, X_i$ and a tuple $\bigl(f_i^j \bigl)_{(i,j) \in I\times I}$ of functions $f_i^j : X_i \to X_i$ such that for any $i$, $j$, $k\in I$:
    \begin{equation}\label{eq:comfunct}
      f_i^j \circ f_i^k = f_i^k \circ f_i^j.
    \end{equation}
    The correspondence is given such that the right P{\l}onka magma structure $\cdot$ on $X$ associated to the pair
    $\bigl( \coprod_{i\in I} \, X_i, \,  \bigl(f_i^j \bigl)_{(i,j) \in I\times I} \bigl)$ is given by:
    \begin{equation}\label{eq:multRPm}
      x_i \cdot x_j := f_i^j (x_i)
    \end{equation}
    for all $x_i \in X_i$, $x_j \in X_j$ and $i$, $j\in I$.\footnote{We note that the multiplication given by \eqref{eq:multRPm} does not depend on the specific element $x_j \in X_j$ chosen, but only on its ``degree" $j$.}

  \item\label{RPmb} $(X, \, \cdot)$ is a right involutory P{\l}onka magma if and only if $f_i^j : X_i \to X_i$ is an involution on $X_i$, for all $i$, $j\in I$;

  \item\label{RPmc} $(X, \, \cdot)$ is a right band P{\l}onka magma if and only if $f_i^{i} = {\rm Id}_{X_i}$, for all $i\in I$;

  \item\label{RPme} $(X, \, \cdot)$ is a $2$-cyclic magma if and only if $f_i^j : X_i \to X_i$ is an involution and $f_i^{i} = {\rm Id}_{X_i}$,
    for all $i$, $j\in I$.
  \end{enumerate}
\end{theorem}

\begin{proof}
  First of all, using \eqref{eq:comfunct}, it is easy to prove that the multiplication defined by \eqref{eq:multRPm} on the set $X = \coprod_{i\in I} \, X_i$ is a right P{\l}onka magma structure. Conversely, let $(X, \cdot)$ be a right P{\l}onka magma. We define the congruence $\approx$ on $X$ as follows: $a \approx b$ if and only if $x\cdot a = x \cdot b$, for all $x\in X$. Then, $\approx$ is an equivalence relation on $X$ and, thanks to the right reduction law, is also a congruence. Let $(a_i)_{i\in I} \subseteq X$ be a system of representatives of the equivalence relation $\approx$. Then the set of all equivalence classes $\{X_i := \overline{a_i} \, | \, i\in I \}$ is a partition on $X$ and define for any $x_i \in X_i$ and $x_j \in X_j$, $f_i^j (x_i) := x_i \cdot x_j \in X_i$. The rest of the proof follows from compatibility conditions \eqref{eq:rkm1}. We also note that the quotient magma $(X/ \approx, \, \cdot)$, as given by \eqref{eq:magfactor}, is the left-zero semigroup: $\overline{a_i} \cdot \overline{a_j} = \overline{a_i \cdot a_j} = \overline{f_i^j(a_i)} = \overline{a_i}$, since $f_i^j(a_i) \in X_i = \overline{a_i}$, for all $i$, $j\in I$.

  The other statements are obtained directly from the definitions of the respective concepts, formulated using the multiplication as defined by \eqref{eq:multRPm}.
\end{proof}

\begin{remarks}\label{res:pl1000}
  \begin{enumerate}[(1)]

  \item\label{item:reducedRP} The correspondence established in Theorem~\ref{th:structRPm}(\ref{item:RPma}) is surjective but \emph{not} injective: different data $\bigl((X_i)_{i\in I},(f_i^j)\bigr)$ may yield the same right P{\l}onka magma via \eqref{eq:multRPm}. For instance, consider $X$ to be a two-element set endowed with the multiplication $x\cdot y=x$ for all $x,y\in X$. It can be easily seen that this magma is already of the form $(X,\cdot_f)$ for $f=\id_X$, but it can also be partitioned into two singletons with all the $f_i^j$ taken to be identities. However, among the various partitions, we distinguish two which are uniquely determined, namely the coarsest and the finest partitions. These are characterized in Theorem \ref{th:uniqpart} below.


  \item\label{strplzero} The structure of $2$-cyclic magmas from Theorem \ref{th:structRPm}(\ref{RPme}) was indicated by P{\l}onka: see the remark preceding \cite[Theorem 4]{plonka2}. In this case, each component $X_i$ of the partition $X = \coprod_{i\in I} \, X_i$ is the left-zero semigroup: indeed, for any $x_i$ and $x_i' \in X_i$, we have $x_i \cdot x_i' = f_i^i (x_i) = x_i$.

    On the other hand, the structure of right band P{\l}onka magmas from Theorem \ref{th:structRPm}(\ref{RPmc}) is exactly \cite[Theorem 2.2]{romanowska}.

  \item\label{descinde} In the decomposition given by Theorem \ref{th:structRPm} we observe that each component $X_i \subseteq X$ of the partition is a right ideal
    of the magma $(X, \cdot)$ and has the form $(X_i, \, \cdot_{f_i^i})$ as given in Example \ref{exs:exerkm}(\ref{generic}), for a self-function on $X_i$.

    On the other hand, the same decomposition is singleton (that is $|I| = 1$) if and only if any two elements $a$, $b\in X$ are congruent in the sense of the proof of Theorem \ref{th:structRPm}; this is easily shown to be equivalent to the existence of a self-function $f: X \to X$ such that $(X, \, \cdot) = (X, \, \cdot_f)$, as constructed in Example \ref{exs:exerkm}(\ref{generic}).

\item\label{strucleftpm} A magma $(X, \, \ast)$ is a left P{\l}onka magma if and only if the opposite magma $(X, \, \ast_{\rm op})$ is a right P{\l}onka magma, where $x \ast_{\rm op} y := y\ast x$, for all $x$, $y\in X$. Thus, we have an isomorphism of categories ${}^{\rm lP}{\rm Mag} \cong {\rm Mag}^{\rm rP}$. The structure theorem of left P{\l}onka magmas follows now from this observation and Theorem \ref{th:structRPm}(\ref{item:RPma}). We indicate it because we will use it as well.

Giving a left P{\l}onka magma structure $\ast$ on a set $X$ is equivalent to specifying a partition $X = \coprod_{i\in I} \, X_i$ and a tuple
$\bigl(g_i^j \bigl)_{(i,j) \in I\times I}$ of functions $g_i^j : X_i \to X_i$ such that for any $i$, $j$, $k\in I$:
 \begin{equation}\label{comfunctleft}
    g_i^j \circ g_i^k = g_i^k \circ g_i^j.
     \end{equation}
The correspondence is given such that the left  P{\l}onka magma structure $\ast$ on $X$ associated to the pair
    $\bigl(\coprod_{i\in I} \, X_i, \,  \bigl(g_i^j \bigl)_{(i,j) \in I\times I} \bigl)$ is given by:
    \begin{equation}\label{multRPmleft}
      x_i \ast x_j := g_j^i (x_j)
    \end{equation}
    for all $x_i \in X_i$, $x_j \in X_j$ and $i$, $j\in I$.
  \end{enumerate}
\end{remarks}


\begin{definition}\label{def:coarsefine}
  Let $(X,\cdot)$ be a right P\l onka magma.
  \begin{enumerate}[(1)]
  \item A structure $(X_i,\ f_i^j)$ as in Theorem \ref{th:structRPm}(\ref{item:RPma}) of $(X,\cdot)$ is called a {\it P\l onka partition}.

  \item Given two P\l onka partitions
    \begin{equation}\label{eq:xny}
      \cX:=(X_i,\ f_{i}^j)
      \quad\text{and}\quad
      \cY:=(Y_{\alpha},\ f_{\alpha}^{\beta})
    \end{equation}
    of a right P\l onka magma we say that $\cX$ is {\it coarser} than $\cY$, or that $\cY$ is {\it finer} than $\cX$ (written $\cX\preceq \cY$ or $\cY\succeq\cX$) if every $Y_{\alpha}$ is contained in some $X_i$ and
    \begin{equation*}
      f_{\alpha}^{\beta} = f_i^j|_{Y_{\alpha}}
      \quad\text{for}\quad
      Y_{\alpha}\subseteq X_i,\ Y_{\beta}\subseteq X_j.
    \end{equation*}
  \end{enumerate}
\end{definition}
We also introduce some language that builds on \cite[Theorem 3.5]{acm1}:

\begin{definition}\label{def:connfunc}

  \begin{enumerate}[(1)]

  \item\label{item:con3} A family $\cF$ of functions $X\to X$ is \emph{connected} if there is no non-trivial partition $X = \coprod_{i\in I} \, X_i$ (i.e. one in at least 2 non-empty parts) such that $f(X_i) \subseteq X_i$, for all $f\in \cF$ and $i\in I$. This is a variant of the notion of connectedness introduced in \cite[\S 1, p.323]{MR214529}.

  \item\label{item:conncomp} The {\it connected components} of a family $\cF$ (or the {\it $\cF$-connected components}) are the maximal $\cF$-invariant subsets of $X$ on which the restriction of $\cF$ is connected.

    For every family $\cF$ there is a partition of $X$ into $\cF$-connected components.

  \item\label{item:con1} A connected singleton $\{f: X \to X \}$ is called a \emph{connected} map.

  \item\label{item:bi-conn} A map $R: X\times X \to X\times X$ is called \emph{bi-connected} if there is no non-trivial partition $X = \coprod_{i\in I} \, X_i$ such that $R (X_i \times X_i) \subseteq X_i \times X_i$, for all $i\in I$.

  \end{enumerate}
\end{definition}

\begin{remarks}\label{re:biconevsinde}
  \begin{enumerate}[(1)]

  \item\label{netrivial} Passing to the equivalent language of bi-magmas, a map $R : X\times X \to X\times X$ is bi-connected if and only if the associated bi-magma
    $(X, \, \cdot, \, \ast)$ cannot be decomposed as an arbitrary non-trivial partition consisting of sub-bi-magmas. We point out that any bi-connected map $R: X\times X \to X\times X$ is indecomposable in the sense of \cite[Definition 2.5]{etingof}. Here, since the YB-solutions considered in this paper are in general degenerate, we use \cite[Definition 2.5]{etingof} without its non-degeneracy requirements: $R$ is called \emph{decomposable} if $X = Y\sqcup Z$ for some non-empty subsets $Y$, $Z$ such that $R(Y\times Y)\subseteq Y\times Y$ and $R(Z\times Z)\subseteq Z\times Z$, and \emph{indecomposable} otherwise. For YB-solutions induced by P{\l}onka bi-magmas the converse implication holds as well, i.e. bi-connectedness and indecomposability are equivalent (see Proposition~\ref{pr:biconn_indec} below). If we assume, as in \cite{etingof}, that the two subsets $Y$ and $Z$ be non-degenerate (that is, right-left non-degenerate in the YB-setting, see Remark~\ref{re:nede-quasigr}), then for degenerate solutions indecomposability becomes strictly weaker than bi-connectedness: for instance, the Lyubashenko solution $R_{(f,\, {\rm Id}_X)}$ on $X=\{a, b, c\}$, where $f(a)=f(b)=a$ and $f(c)=c$, is not bi-connected, since $X=\{a, b\}\sqcup\{c\}$ is a partition into invariant subsets, but this is the only decomposition of $X$ into two invariant subsets and the restriction of $R$ to $\{a, b\}$ is degenerate. On the other hand, if $X$ is finite and $R$ is right-left non-degenerate, then the restriction of $R$ to any invariant subset is again right-left non-degenerate (an injective self-map of a finite set is bijective), so in this case the two versions of indecomposability coincide.

  \item It is easy to show that a map $f:X \to X$ is connected in the sense of Definition \ref{def:connfunc}(\ref{item:con1}) if and only if there is no non-trivial $f$-invariant partition of $X$, i.e. $X = Y\sqcup Z$ for non-empty subsets $Y$, $Z$ such that $f(Y) \subseteq Y$ and $f(Z) \subseteq Z$.

    For a positive integer $n$, the number of all conjugacy classes of all connected self-maps $f: \{1, \cdots, n\} \to \{1, \cdots, n\}$ is called the \emph{Harary number} of $n$ \cite[Definition 3.6 and Remarks 3.7]{acm1} and is denoted by $\mathfrak{c} (n)$.

  \item\label{bic_self} A pair $(f, \, g)$ of self-functions $f$, $g : X \to X$ is connected in the sense of Definition \ref{def:connfunc}(\ref{item:con3}) if and only if $R: = f \times g$ is bi-connected.
  \end{enumerate}
\end{remarks}

\begin{theorem}\label{th:uniqpart}
  Let $(X,\cdot)$ be a right P\l onka magma.

  \begin{enumerate}[(1)]
  \item The P\l onka partition constructed in the proof of Theorem \ref{th:structRPm} is the coarsest, and is unique.

  \item There is also a finest P\l onka partition, uniquely characterized as one for which the families
    \begin{equation*}
      \cF_i:=\left\{X_i\xrightarrow{\quad f_i^j\quad} X_i\right\}_{j\in I}
    \end{equation*}
    are all connected.
  \end{enumerate}
\end{theorem}
\begin{proof}
  In both cases, uniqueness follows immediately from the characterization as the coarsest/finest P\l onka partition: `finer than' is plainly an order, and a partially ordered set has at most one smallest and/or largest element.

  \begin{enumerate}[(1)]
  \item In {\it any} P\l onka partition $(X_i, f_i^j)$ elements $a,b\in X_i$ of the same class have the property that $x\cdot a = x\cdot b$ for all $x\in X$. Since this is precisely the equivalence relation that induces the partition in the proof of Theorem \ref{th:structRPm}, that partition must be the coarsest achievable.

  \item For an arbitrary P\l onka partition, any refinement thereof (i.e. passage to a finer one) will decompose some $X_i$ into parts left invariant by the family $\{f_i^j\}_j$. This means that a partition for which said families are all connected cannot be refined {\it strictly}, so that such a P\l onka partition (henceforth also referred to as {\it connected}) is maximal with respect to the `$\preceq$' order.

    Note, next, that every P\l onka partition $(X_i,f_i^j)$ is refinable into a connected one: indeed, one can simply define the smaller parts $X_{i,\alpha}$ to be the respective connected components of $\{f_i^j\}_j$, with $f_{i,\alpha}^{j,\beta}$ obtained by restricting $f_i^j$ in the obvious fashion.

    This, so far, says that the connected P\l onka partitions are precisely the maximal elements with respect to $\preceq$. We still have to argue that there is in fact only {\it one} such partition. This follows from the more general remark that the poset
    \begin{equation*}
      \left(\text{P\l onka partitions},\preceq\right)
    \end{equation*}
    is {\it filtered} \cite[\S IX.1]{mcl}: every two elements \eqref{eq:xny} have a common refinement whose underlying set partition consists simply of the pairwise intersections $X_i\cap Y_{\alpha}$.   This concludes the proof.

  \end{enumerate}
\end{proof}
The uniqueness of the two extreme P\l onka partitions of Theorem \ref{th:uniqpart} also makes them into complete invariants for the right-P\l onka-magma structure:

\begin{theorem}\label{th:isopart}
  For two right P\l onka magmas $(X,\cdot)$ and $(Y,\diamond)$ the following conditions are equivalent:
  \begin{enumerate}[(1)]
  \item\label{item:rpiso} $(X,\cdot)$ and $(Y,\diamond)$ are isomorphic.

  \item\label{item:rpisocoarse} There is a bijection
    \begin{equation*}
      \{i\}\xrightarrow[\cong]{\quad\theta\quad}\{\alpha\}
    \end{equation*}
    between the indexing sets of the respective coarsest P\l onka partitions of $(X,\cdot)$ and $(Y,\diamond)$ and isomorphisms
    \begin{equation*}
      \left\{f_i^j\right\}_j
      \xrightarrow[\cong]{\quad}
      \left\{g_{\theta(i)}^{\theta(j)}\right\}_{\theta(j)}
    \end{equation*}
    as families of self-functions.

  \item\label{item:rpisofine} Same as (\ref{item:rpisocoarse}), for the {\it finest} P\l onka partitions instead.
  \end{enumerate}
\end{theorem}
\begin{proof}
  This follows easily by observing that isomorphisms as in either (\ref{item:rpisocoarse}) or (\ref{item:rpisofine}) will induce isomorphisms of magmas given by \eqref{eq:multRPm}. Conversely, an isomorphism of magmas induces one on the posets of P\l onka partitions, hence also respective isomorphisms between the smallest and largest elements of those posets.
\end{proof}
\subsection{Ideal-simple right P{\l}onka magmas}\label{sse:idsimpright}
Recall that a magma $(X, \cdot)$ is called \emph{right} (resp. \emph{left}) \emph{ideal-simple} if it has no proper non-empty right (resp. left) ideals, and \emph{ideal-simple} if it has no proper non-empty two-sided ideals.

\begin{example}\label{ex:cyclelX}
Let $f: X \to X$ be a map and consider the magma $(X, \, \cdot_f)$ as defined in Example \ref{exs:exerkm}(\ref{generic}). Then,
$I \subseteq X$ is a left ideal of $(X, \, \cdot_f)$ if and only if $I \supseteq f(X)$. Furthermore, $I$ is a two-sided ideal of
$(X, \, \cdot_f)$ if and only if $I$ is a left ideal. Thus, $(X, \, \cdot_f)$ is a (left) ideal-simple magma if and only if $f: X \to X$ is a surjection.\\
On the other hand, $I \subseteq X$ is a right ideal of $(X, \, \cdot_f)$ if and only if $I$ is $f$-invariant, i.e. $f(I) \subseteq I$.
Thus, we obtain that $(X, \, \cdot_f)$ is a right ideal-simple magma if and only if $X$ is the only $f$-invariant subset of $X$, that is $f$ is a
 $|X|$-cycle (or an incompressible map) as defined by Definition \ref{def:simple}.
\end{example}
As a first application of Theorem \ref{th:structRPm} the structure of the (left) ideal-simple P{\l}onka magmas now follows:

\begin{theorem}\label{th:lsimp}
  For a right P{\l}onka magma $(X, \, \cdot)$ decomposed as in Theorem \ref{th:structRPm}(\ref{item:RPma}) the following conditions are then equivalent:
  \begin{enumerate}[(1)]
  \item\label{item:lsimp} $(X, \, \cdot)$ is left ideal-simple.
  \item\label{item:simp} $(X, \, \cdot)$ is ideal-simple.
  \item\label{item:tot} $(X, \, \cdot)$ is {\it total}, in the sense that $X^2=X$.
  \item\label{item:jonto} For every $i\in I$ the functions $X_i\xrightarrow{f_i^j}X_i$, $j\in I$ are {\it jointly onto}, in the sense that their images cover $X_i$.
  \end{enumerate}
\end{theorem}

\begin{proof}
  \begin{enumerate}[label={}]

  \item {\bf (\ref{item:lsimp}) $\Longrightarrow$ (\ref{item:simp}) $\Longrightarrow$ (\ref{item:tot}).} This is obvious: $X^2$ is a two-sided ideal (non-empty if $X$ is), and every two-sided ideal is also a left ideal.

  \item {\bf (\ref{item:tot}) $\Longleftrightarrow$ (\ref{item:jonto}).} The multiplication defined by \eqref{eq:multRPm} makes it clear that for every $i\in I$ we have
    \begin{equation*}
      X^2\cap X_i = \bigcup_{j\in I}\mathrm{Im}~f_i^j,
    \end{equation*}
    so indeed the right-hand side is all of $X_i$ for every $i\in I$ precisely when
    \begin{equation*}
      X^2\supseteq X_i,\ \forall i\in I
      \iff
      X^2=X.
    \end{equation*}

  \item {\bf (\ref{item:jonto}) $\Longrightarrow$ (\ref{item:lsimp}).} Let $\emptyset\ne Y\subseteq X$ be a left ideal. By \eqref{eq:multRPm} left multiplication by elements in $X_i$ produces elements in $X_i$, so
    \begin{equation*}
      Y\cap X_i\ne \emptyset,\ \forall i\in I.
    \end{equation*}
    Now, an arbitrary $x_i\in X_i$ lies in the image of some $f_i^j$, $j\in I$ by assumption, and we have just observed that there is some $x_j\in Y\cap X_j$. We thus have
    \begin{equation*}
      x_i\in \mathrm{Im}~f_i^j = X_i\cdot \{x_j\}\subseteq X_i\cdot Y\subseteq Y
    \end{equation*}
    by \eqref{eq:multRPm}, and we are done. This completes the implication circle, and hence the proof.
  \end{enumerate}
\end{proof}

\begin{remark}
  Note that the condition of Theorem \ref{th:lsimp}(\ref{item:jonto}) is strictly weaker than the requirement that at least one $f_i^j$, $j\in I$ be onto. To this end, in the decomposition of Theorem \ref{th:structRPm}(\ref{item:RPma}) set $I:=\bZ$, $X_i=\bZ$ and furthermore
  \begin{equation*}
    X_i=\bZ\ni x\xmapsto{\quad f_i^j\quad}\min(x,j)\in \bZ=X_i.
  \end{equation*}
  for all $i$, $j \in \bZ$. Clearly, no {\it individual} $f_i^j$ is onto, but the entire family $\{f_i^j\}_{j\in \bZ}$ is {\it jointly} onto in the sense of Theorem \ref{th:lsimp}(\ref{item:jonto}).
\end{remark}
As a second application of Theorem \ref{th:structRPm}(\ref{item:RPma}) we classify the right ideal-simple P{\l}onka magmas. By contrast to Theorem \ref{th:lsimp}, these turn out
to be finite and unique up to isomorphism given the cardinality $|X|$.

\begin{theorem}\label{th:simplestr}
  A right P{\l}onka magma $(X, \cdot)$ is right ideal-simple if and only if $X$ is a finite set and there exists an $|X|$-cycle $f: X \to X$ such that $(X, \cdot) = (X, \, \cdot_f)$, as given in Example \ref{exs:exerkm}(\ref{generic}).

  Up to isomorphism, there is only one right ideal-simple right P{\l}onka magma on a finite set $X$ with $|X| = n$, namely the one associated to a cycle of length $n$ in the permutation group $S_n$.
\end{theorem}
\begin{proof}
  We have seen in Example \ref{ex:cyclelX} that, for a function $f: X \to X$, the magma $(X, \, \cdot_f)$ is right ideal-simple if and only if $f: X \to X$ is an $|X|$-cycle. By Lemma \ref{le:cycles} this happens if and only if $X$ is finite and $f$ is a translation by a generator of the finite cyclic group structure on $X$.

  Now, let $(X, \cdot)$ be a right P{\l}onka magma that is right ideal-simple. Then it has a decomposition $X = \coprod_{i\in I} \, X_i$ as given by
  Theorem \ref{th:structRPm}(\ref{item:RPma}). Since the multiplication is given by \eqref{eq:multRPm} we obtain that each component $X_i$ of the partition is a right ideal of $(X, \cdot)$. As $X$ is right ideal-simple, we obtain that $|I| = 1$ and hence using the last observation of Remark \ref{res:pl1000}(\ref{descinde}), there exists a self-function $f: X \to X$ such that $(X, \, \cdot) = (X, \, \cdot_f)$. This proves the first statement. Moreover, we note that in this case $(X, \, \cdot_f)$ is also left ideal-simple.

  Finally, we observe that for two maps $f$, $g: X \to X$, the magmas $(X, \, \cdot_f)$ and $(X, \, \cdot_g)$ are isomorphic if and only if $f$ and $g$ are conjugate, i.e. there exists $\sigma \in \Sigma_X$, a permutation on $X$, such that $g = \sigma \circ f \circ \sigma^{-1}$. This proves the last statement since any two cycles of length $n$ from the permutation group $S_n$ are conjugate.
\end{proof}
\subsection{The structure of P{\l}onka bi-magmas and non-degenerate YB-solutions}\label{sse:strbi}
Next we give the structure theorem of P{\l}onka bi-magmas: it shows that any
P{\l}onka bi-magma can be decomposed as a partition of Lyubashenko's sub-bi-magmas as defined in Example \ref{exs:bmlyb}(\ref{item:LB}).

\begin{theorem}\label{th:strpbimag}
  Let $X$ be a set. A bi-magma structure $(X, \, \cdot, \, \ast)$ on $X$ is a P{\l}onka bi-magma if and only if there exist:

  \begin{itemize}

  \item A  partition $(X_i)_{i\in I}$ on $X$, i.e. $X = \coprod_{i\in I} \, X_i$;

  \item A pair of tuples $\bigl(f_i^j \bigl)_{(i,j) \in I\times I}$ and $\bigl(g_i^j \bigl)_{(i,j) \in I\times I}$
    of functions $f_i^j : X_i \to X_i$ and $g_i^j : X_i \to X_i$ such that for any $i$, $j$, $k\in I$:
    \begin{equation}\label{eq:comfunctddd}
      f_i^j \circ f_i^k = f_i^k \circ f_i^j, \qquad g_i^j \circ g_i^k = g_i^k \circ g_i^j, \qquad g_i^j \circ f_i^k = f_i^k \circ g_i^j.
    \end{equation}
  \end{itemize}

  The P{\l}onka bi-magma $(X, \, \cdot, \, \ast)$ associated to a partition
  $\bigl((X_i)_{i\in I}, \,  \bigl(f_i^j \bigl)_{(i,j) \in I\times I}, \, \bigl(g_i^j \bigl)_{(i,j) \in I\times I} \bigl)$ as described above is given by:
  \begin{equation}\label{eq:multRPmdd}
    x_i \cdot x_j := f_i^j (x_i), \qquad   x_i \ast x_j := g_j^i (x_j)
  \end{equation}
  for all $x_i \in X_i$, $x_j \in X_j$ and $i$, $j\in I$. Furthermore, the associated P{\l}onka bi-magma is unitary if and only if
  $f_i^j$ and $g_i^j$ are bijective functions inverse to each other, for all $i$, $j\in I$.
\end{theorem}
\begin{proof}
  Assume first that we are given data $\bigl((X_i)_{i\in I}, \, \bigl(f_i^j \bigl)_{(i,j) \in I\times I}, \, \bigl(g_i^j \bigl)_{(i,j) \in I\times I} \bigl)$ satisfying \eqref{eq:comfunctddd}. Then it is a routine matter to check that the set $X = \coprod_{i\in I} \, X_i$ with the multiplications defined by \eqref{eq:multRPmdd} is a P{\l}onka bi-magma.

  As for the converse we adapt the proof of Theorem \ref{th:structRPm} as follows. Let $(X, \, \cdot, \, \ast)$ be a P{\l}onka bi-magma. On the set $X$ we define a relation $\rho$ as follows: $a \rho b$ if and only if $x \cdot a = x \cdot b$ and $a \ast x = b \ast x$, for all $x\in X$. Then, $\rho$ is an equivalence relation on $X$. Moreover, using the right reduction law of $\cdot$, the left reduction law of $\ast$ and the compatibility conditions \eqref{eq:KM3b} and \eqref{eq:KM3c} we can easily prove that $\rho$ is a congruence on $(X, \, \cdot, \, \ast)$, i.e. $\rho$ is compatible with both multiplications $\cdot$ and $\ast$ such that we can construct the quotient bi-magma $X/ \rho$.

  Let $(a_i)_{i\in I} \subseteq X$ be a system of representatives of the congruence $\rho$. Then the set of all congruence classes $\{X_i := \overline{a_i} \, | \, i\in I \}$ is a partition on $X$ and define: $f_i^j (x_i) := x_i \cdot x_j \in X_i$ and $g_i^j (x_i) := x_j \ast x_i \in X_i$, for all $x_i \in X_i$, $x_j \in X_j$ and $i$, $j\in I$. The rest of the proof follows from compatibility conditions defining the P{\l}onka bi-magma as given in Definition \ref{def:kbimag}.

  We also note that the quotient bi-magma $(X/ \rho, \, \cdot, \, \ast)$ is given such that $(X/ \rho, \, \cdot)$ is the left-zero semigroup and $(X/ \rho, \, \ast)$ is the right-zero semigroup, i.e. $\overline{a_i} \cdot \overline{a_j} = \overline{a_i}$ and $\overline{a_i} \ast \overline{a_j} = \overline{a_j}$, for all $i$, $j\in I$.

  Finally, we can easily check that the unitary condition \eqref{eq:uniKbimag} comes down to $f_j^i \circ g_j^i = {\rm Id}_{X_j}$, for all $i$, $j\in I$. Hence, using the last compatibility of \eqref{eq:comfunctddd}, this reduces to the fact that $f_i^j$ and $g_i^j$ are bijective functions inverse to each other, for all $i$, $j\in I$ and this finishes the proof.
\end{proof}

\begin{remark}\label{re:reduceddata}
The correspondence established in Theorem~\ref{th:strpbimag} is surjective but \emph{not} injective: different data $\bigl((X_i)_{i\in I},(f_i^j),(g_i^j)\bigr)$ may give rise to the same P{\l}onka bi-magma via \eqref{eq:multRPmdd}. For example, consider $X :=\{a, b\}$ to be a two-element set. Then both ($X_1 =X$, $f_1^1 = g_1^1 = {\rm Id}_X$) and $\bigl(X_1 = \{a\}$, $X_2 = \{b\},\, (f_i^j),\, (g_i^j)\bigl)$, where all $f_i^j$, $g_i^j$ are identities, satisfy~(\ref{eq:comfunctddd}) and yield the same P{\l}onka bi-magma defined by $x\cdot y=x$ and $x\ast y=y$ for all $x$, $y \in X$.
\end{remark}
By analogy to Definition \ref{def:coarsefine} we introduce:

\begin{definition}\label{def:biplonka}
  Let $(X,\cdot,*)$ be a P\l onka bi-magma.
  \begin{enumerate}[(1)]
  \item A tuple $(X_i,\ f_i^j,\ g_i^j)$ as in Theorem \ref{th:strpbimag} is called a {\it P\l onka bi-partition}.

  \item Given two P\l onka bi-partitions
    \begin{equation*}
      \cX:=(X_i,\ f_{i}^j,\ g_i^j)
      \quad\text{and}\quad
      \cY:=(Y_{\alpha},\ f_{\alpha}^{\beta},\ g_{\alpha}^{\beta})
    \end{equation*}
    of $(X,\cdot,*)$ we say that $\cX$ is {\it coarser} than $\cY$, or that $\cY$ is {\it finer} than $\cX$ (written $\cX\preceq \cY$ or $\cY\succeq\cX$) if every $Y_{\alpha}$ is contained in some $X_i$ and
    \begin{equation*}
      f_{\alpha}^{\beta} = f_i^j|_{Y_{\alpha}}
      ,\quad
      g_{\alpha}^{\beta} = g_i^j|_{Y_{\alpha}}
      \quad\text{for}\quad
      Y_{\alpha}\subseteq X_i,\ Y_{\beta}\subseteq X_j.
    \end{equation*}
  \end{enumerate}
\end{definition}

\begin{theorem}\label{non-deg_lr+rl}
Let $(X, \, \cdot, \, \ast)$ be a P{\l}onka bi-magma and consider $R = R_{(\cdot, \, \ast)}$ to be the associated YB-solution given by (\ref{eq:solgenK}).
 \begin{enumerate}[(1)]
  \item\label{item:sing2} $R$ is right-left non-degenerate if and only if there exists a P\l onka bi-partition $(X_i,\ f_i^j,\ g_i^j)$ of $(X, \, \cdot, \, \ast)$
  such that $f_i^j$, $g_i^j : X_i \to X_i$ are bijective for all $i$, $j \in I$. In particular, $R$ is right-left non-degenerate if and only if it is bijective.
    \item\label{item:sing} $R$ is left-right non-degenerate if and only if $X = \{ \star \}$ is a singleton set.
    \end{enumerate}
\end{theorem}
\begin{proof}
 \begin{enumerate}[(1)]
  \item Let $(X_i,\ f_i^j,\ g_i^j)$ be a P\l onka bi-partition of $(X, \, \cdot, \, \ast)$ as in Theorem~\ref{th:strpbimag}. In particular, we have
   $$
    x_i \cdot x_j := f_i^j (x_i), \qquad   x_i \ast x_j := g_j^i (x_j)\,\,\,\, {\rm for\,\, all}\,\, x_i \in X_i, \, x_j \in X_j \,\, {\rm and} \,\, i,\, j \in I.
$$
 Now $(X,\, \cdot)$ is a right quasi-group if and only if $r_{a}^{\cdot} : X \to X$, $r_{a}^{\cdot}(x) = x \cdot a$ is bijective for all $a \in X$. To start with, note that for a fixed $a \in X$, the map $r_{a}^{\cdot}$ preserves each component of the P\l onka bi-partition, i.e. $r_{a}^{\cdot}(X_{i}) \subseteq X_{i}$ for all $i \in I$. Indeed, let $x= x_{i} \in X_{i}$; we have
  $$
  r_{a}^{\cdot}(x_{i}) = x_{i} \cdot a = f_{i}^{j}(x_{i}) \in X_{i}
  $$
  where $j$ is the unique element of $I$ such that $a \in X_{j}$. This proves that $r_{a}^{\cdot}$ is bijective if and only if $f_{i}^{j}$ is bijective for all $i \in I$. Now recall that $(X,\, \cdot)$ being a right quasi-group is equivalent to $r_{a}^{\cdot} : X \to X$ being bijective for all $a \in X$ which, in turn, as $a$ ranges over all $X_{j}$'s, is equivalent to $f_{i}^{j}$ being bijective for all $i$, $j \in I$.

 Similarly, $(X,\, \ast)$ is a left quasi-group if and only if $l_{a}^{\ast} : X \to X$, $l_{a}^{\ast}(x) = a \ast x$ is bijective for all $a \in X$. We will prove that this is equivalent to $g_j^i$ being bijective for all $i$, $j \in I$. As in the case of the right translation map, for a fixed $a \in X$, the map $l_{a}^{\ast}$ preserves each component of the P\l onka bi-partition, i.e. $l_{a}^{\ast}(X_{i}) \subseteq X_{i}$ for all $i \in I$. Indeed, let $x= x_{i} \in X_{i}$ for some $i \in I$. We have
  $$
  l_{a}^{\ast}(x_{i}) = a \ast x_{i}  = g_{i}^{k}(x_{i}) \in X_{i}
  $$
  where $k$ is the unique element of $I$ such that $a \in X_{k}$. This proves that $l_{a}^{\ast}$ is bijective if and only if $g_{i}^{k}$ is bijective for all $i \in I$. Now recall that $(X,\, \ast)$ being a left quasi-group is equivalent to $l_{a}^{\ast} : X \to X$ being bijective for all $a \in X$ which, in turn, as $a$ ranges over all $X_{k}$'s, is equivalent to $g_{i}^{k}$ being bijective for all $i$, $k \in I$. This completes the proof of our first statement.

Again by Theorem~\ref{th:strpbimag}, for all $x_{i} \in X_{i}$ and $x_{j} \in X_{j}$ we have:
$$
R(x_{i},\, x_{j}) = (x_{i} \cdot x_{j},\, x_{i} \ast x_{j}) = (f_{i}^{j}(x_{i}),\, g_{j}^{i}(x_{j})) \in X_{i} \times X_{j}
$$
This shows that $R$ preserves each subset $X_{i} \times X_{j}$ of the partition of $X \times X$ induced by the given partition of $X$ and, moreover, the restriction of $R$ to each such subset is $f_{i}^{j} \times g_{j}^{i}$. Therefore, $R$ is bijective if and only if all its restrictions to the subsets of this partition are bijective. This comes down to $R$ being bijective if and only if all maps $f_{i}^{j}$, $g_{j}^{i}$ are bijective. By the first part of the proof this is equivalent to $R$ being right-left non-degenerate, as desired.
    \item Assume first that $R$ is left-right non-degenerate. It then follows from \eqref{eq:KM3c} and \eqref{eq:KM3b} that $y \ast z = z$ and $x \cdot z = x$, for all $x, y, z \in X$. Hence, $R = {\rm Id}_{X^2}$, which is left-right non-degenerate if and only if $X = \{ \star \}$ is a singleton. The converse is obvious.
    \end{enumerate}
\end{proof}
The coarseness/fineness discussion carries over from right P\l onka magmas to P\l onka bi-magmas with only minimal effort. We therefore omit the proofs of the following straightforward analogues of Theorems \ref{th:uniqpart} and \ref{th:isopart}:

\begin{theorem}\label{th:biuniqpart}
  Let $(X,\cdot,*)$ be a P\l onka bi-magma. Then:

  \begin{enumerate}[(1)]
  \item The P\l onka bi-partition constructed in the proof of Theorem \ref{th:strpbimag} is the coarsest, and is unique.

  \item There is also a finest P\l onka bi-partition, uniquely characterized as one for which the families
    \begin{equation*}
      \cF_i:=\left\{X_i\xrightarrow{\quad f_i^j,\ g_i^j\quad} X_i\right\}_{j\in I}
    \end{equation*}
    are all connected.  \qedhere
  \end{enumerate}
\end{theorem}

\begin{theorem}\label{th:biisopart}
  For two P\l onka bi-magmas $(X,\cdot,*)$ and $(\overline{X},\overline{\cdot},\overline{*})$ the following conditions are equivalent:
  \begin{enumerate}[(1)]

  \item\label{item:birpiso} $(X,\cdot,*)$ and $(\overline{X},\overline{\cdot},\overline{*})$ are isomorphic.

  \item\label{item:birpisocoarse} There is a bijection
    \begin{equation*}
      \{i\}\xrightarrow[\cong]{\quad\theta\quad}\{\alpha\}
    \end{equation*}
    between the indexing sets of the respective coarsest P\l onka bi-partitions of $(X,\cdot,*)$ and $(\overline{X},\overline{\cdot},\overline{*})$ and isomorphisms between the families
    \begin{equation*}
      \left\{f_i^j,\ g_i^j\right\}
      \quad\text{and}\quad
      \left\{
        \overline{f}_{\theta(i)}^{\theta(j)}
        ,\
        \overline{g}_{\theta(i)}^{\theta(j)}
      \right\}
    \end{equation*}
    as families of self-functions.

  \item\label{item:birpisofine} Same as (\ref{item:birpisocoarse}), for the {\it finest} P\l onka bi-partitions instead.  \qedhere
  \end{enumerate}
\end{theorem}
As a consequence of Theorem~\ref{th:strpbimag}, Theorem~\ref{non-deg_lr+rl}(\ref{item:sing2}) and Theorem~\ref{th:biisopart}, we obtain the following characterization of the right-left
non-degenerate induced YB-solutions.

\begin{corollary}\label{thm:class_R_L_Non}
The right-left non-degenerate YB-solutions induced by P{\l}onka bi-magmas are precisely the maps
\begin{equation}
R(x, \, y) = \Bigl(f^{\beta(y)}_{\beta(x)}(x), \,\, g^{\beta(x)}_{\beta(y)}(y)\Bigr),
\end{equation}
where $\beta:X\to I$ is the block map of a partition $X=\coprod_{i\in I}X_i$ and, for every $i\in I$,
the family $\{f_i^j,\ g_i^j\ :\ j\in I\}\subseteq\Sigma_{X_i}$ consists of pairwise commuting
permutations of $X_i$. \\
Two such solutions are isomorphic in $\mathcal{YB}$ if and only if, at the level of the coarsest P{\l}onka bi-partitions, there exist bijections $\theta:I\to I'$ and $\sigma_i:X_i\to X'_{\theta(i)}$ such that
$$
\sigma_i f_i^j\sigma_i^{-1}=f'^{\,\theta(j)}_{\theta(i)},\qquad
\sigma_i g_i^j\sigma_i^{-1}=g'^{\,\theta(j)}_{\theta(i)}
$$
for all  $i, j\in I$. \qedhere
\end{corollary}

\begin{remark}\label{re:bijunitinv}
Let $(X, \, \cdot, \, \ast)$ be a P{\l}onka bi-magma, $(X_i,\ f_i^j,\ g_i^j)$ a P{\l}onka bi-partition of it and $R = R_{(\cdot, \, \ast)}$. As observed in the proof of Theorem~\ref{non-deg_lr+rl}, $R$ maps each $X_i\times X_j$ into itself and its restriction to $X_i\times X_j$ is $f_i^j\times g_j^i$; similarly, $R^{21}(x_i, \, x_j) = \bigl(g_i^j(x_i), \, f_j^i(x_j)\bigr)$ for all $x_i\in X_i$, $x_j\in X_j$. Consequently, bijectivity, unitarity and involutivity of $R$ can be described in terms of the structural data of Theorem~\ref{th:strpbimag}:
\begin{enumerate}[(a)]
\item $R$ is bijective (equivalently, right-left non-degenerate) if and only if all $f_i^j$ and $g_i^j$ are bijective; in this case $R^{-1}$ restricts to $(f_i^j)^{-1}\times (g_j^i)^{-1}$ on $X_i\times X_j$;
\item $R$ is unitary if and only if $f_i^j$ and $g_i^j$ are bijections inverse to each other, for all $i$, $j\in I$ (cf. Theorems~\ref{th:teorema1} and \ref{th:strpbimag});
\item for any integer $n\geq 1$, we have $R^n = {\rm Id}_{X\times X}$ if and only if $(f_i^j)^n = {\rm Id}_{X_i}$ and $(g_i^j)^n = {\rm Id}_{X_i}$, for all $i$, $j\in I$. In particular, $R$ is involutive if and only if all $f_i^j$ and $g_i^j$ are involutions (cf. Theorem~\ref{th:teorema1} and Theorem~\ref{th:structRPm}(\ref{RPmb})).
\end{enumerate}
Therefore, every unitary or involutive (more generally, every finite order) YB-solution induced by a P{\l}onka bi-magma is bijective, and $R$ is both unitary and involutive if and only if $f_i^j = g_i^j$ is an involution for all $i$, $j\in I$. Conversely, if $X$ is finite and $R$ is bijective, then $R$ has finite order, namely the least common multiple of the orders of the permutations $f_i^j$ and $g_i^j$. Apart from these, there are no further implications between bijectivity, unitarity and involutivity, as the following Lyubashenko solutions $R = R_{(f, \, g)}$ (see Example~\ref{exs:exkimbimag}(\ref{LyPl_can})) show:
\begin{itemize}
\item if $X = \{1, 2\}$, $f = (1\,2)$ and $g = {\rm Id}_X$ then $R$ is involutive but not unitary;
\item if $X = \{1, 2, 3\}$, $f = (1\,2\,3)$ and $g = f^{-1}$ then $R$ is unitary but not involutive;
\item if $X = \{1, 2, 3\}$, $f = (1\,2\,3)$ and $g = {\rm Id}_X$ then $R$ is bijective, but neither unitary nor involutive.
\end{itemize}
Finally, since bijectivity is equivalent to right-left non-degeneracy (Theorem~\ref{non-deg_lr+rl}(\ref{item:sing2})), the bijective YB-solutions induced by P{\l}onka bi-magmas are precisely those described, and classified up to isomorphism, in Corollary~\ref{thm:class_R_L_Non}.
\end{remark}

\subsection{Ideal-simple and bi-connected YB-solutions}\label{sse:idsimpbic}
Following the definition of an ideal in a bi-magma, a non-empty subset $I \subseteq X$ is called a \emph{right-left ideal} of a P{\l}onka bi-magma $(X, \, \cdot, \, \ast)$ if it is simultaneously a right ideal of $(X, \, \cdot)$ and a left ideal of $(X, \, \ast)$. Furthermore, $(X, \, \cdot, \, \ast)$ is called a \emph{right-left ideal-simple} P{\l}onka bi-magma if it has no proper non-empty right-left ideals.

\begin{remark}\label{re:Lybneb}
In the decomposition of a P{\l}onka bi-magma $(X, \, \cdot, \, \ast)$ given by Theorem \ref{th:strpbimag} we observe that each component $X_i \subseteq X$ of the partition is
a right ideal of the magma $(X, \, \cdot)$ and a left ideal of $(X, \, \ast)$, hence it is a right-left ideal of $(X, \, \cdot, \, \ast)$ in the sense of the above definition. In particular, each $X_i$ is a P{\l}onka sub-bi-magma of $(X, \, \cdot, \, \ast)$ and is a Lyubashenko bi-magma, i.e. it has the form $(X_i, \, \cdot_{f_i^i}, \ast^{g_i^i})$ as given in Example \ref{exs:bmlyb}(\ref{item:LB}). Thus, Theorem \ref{th:strpbimag} gives a decomposition of any P{\l}onka bi-magma into subalgebras that are Lyubashenko bi-magmas.

The same decomposition is singleton (that is $|I| = 1$) if and only if any two elements $a$, $b\in X$ are congruent in the sense of the proof of Theorem \ref{th:strpbimag}; this is equivalent to the fact that there exist two commuting self-functions $f$, $g: X \to X$ such that $(X, \, \cdot, \, \ast) = (X, \, \cdot_f, \ast^g)$, i.e.
the P{\l}onka bi-magma is a Lyubashenko bi-magma.
\end{remark}

\begin{example}\label{ex:rlsim}
Let $f$, $g : X \to X$ be two commuting self-functions and $X_f^g = (X, \, \cdot_f, \, \ast^g)$ the associated P{\l}onka bi-magma of Example \ref{exs:exkimbimag}(\ref{LyPl_can}).
Then $X_f^g$ is a right-left ideal-simple bi-magma if and only if $X$ is the only non-empty subset
simultaneously invariant under $f$ and $g$, i.e. $(f, \, g)$ is an incompressible pair of maps as defined in Definition \ref{def:simple}(\ref{item:incomprdef}).
\end{example}
Using Theorem \ref{th:strpbimag} and  Corollary \ref{cor:simpbimag} we are able to classify up to an isomorphism all right-left ideal-simple P{\l}onka bi-magmas:

\begin{theorem}\label{cor:simpbimagb} Let $X$ be a set. Then:
  \begin{enumerate}[(1)]

  \item\label{simp11} A P{\l}onka bi-magma $(X, \, \cdot, \, \ast)$ is right-left ideal-simple if and only if there exists an incompressible pair $(f, \, g)$ of commuting maps
    $f$, $g: X\to X$ such that $(X, \, \cdot, \, \ast) = (X, \, \cdot_f, \, \ast^g)$.

  \item\label{simp12} The isomorphism classes of right-left ideal-simple P{\l}onka bi-magmas (on arbitrary sets) are parameterized by the set of all triples  $(m,\ n,\ d)\in \bZ_{\ge 0}\times \bZ_{>0}^2$, where $d\le |\bZ/m|$. The bijective correspondence is given such that the P{\l}onka bi-magma $(X, \, \cdot_f, \, \ast^g)$ associated to $(m,\ n,\ d)$ is $X_{f}^{g}$, with $(X,\ f,\ g)$ as in Corollary \ref{cor:simpbimag}.

  \item\label{sim100} If $X$ is a finite set with $|X| = t$, then the number of isomorphism classes of all right-left ideal-simple P{\l}onka bi-magmas
    on $X$ is $\sigma (t)$, the sum of all divisors of $t$.
      \end{enumerate}
\end{theorem}

\begin{proof}
  We have seen in Example \ref{ex:rlsim} that, for two commuting self-maps $f$, $g : X \to X$, the P{\l}onka bi-magma $(X, \, \cdot_f, \, \ast^g)$ is right-left ideal-simple if and only if $(f, \, g)$ is an incompressible pair of maps as defined in Definition \ref{def:simple}(\ref{item:incomprdef}).

 Now, let $(X, \, \cdot, \, \ast)$ be a right-left ideal-simple P{\l}onka bi-magma. Then it has a decomposition $X = \coprod_{i\in I} \, X_i$ as given by Theorem \ref{th:strpbimag}. As we observe in Remark \ref{re:Lybneb} each component $X_i$ of the partition is a right-left ideal of $(X, \, \cdot, \, \ast)$. Since $X$ is right-left ideal-simple, we obtain that $|I| = 1$ and hence using the last observation of Remark \ref{re:Lybneb}, there exist two functions $f$, $g : X \to X$ such that $(X, \, \cdot, \, \ast) = (X, \, \cdot_f, \, \ast^g)$. This proves the first statement.

 Part (2) follows from Corollary \ref{cor:simpbimag} which, using Example \ref{exs:exsolsimple}(\ref{doiideal}), in fact classifies all commuting incompressible pairs $(f, \, g)$. Finally, by applying part (\ref{simp12}) we obtain that the number of isomorphism classes of all right-left ideal-simple P{\l}onka bi-magmas of order $t$ is precisely the sum of the divisors of $t$. Indeed, for each positive integer $m | t$, there are exactly $m$ possibilities to choose $d$. This completes the proof.
\end{proof}

\begin{theorem}\label{th:teorema12}
  Let $(X, \, \cdot, \, \ast)$ be a P{\l}onka bi-magma and consider the  canonical associated YB-solution $R = R_{(\cdot, \, \ast)} : X\times X \to X\times X$, $R (x, \, y) = (x\cdot y, \, x\ast y)$. Then:

  \begin{enumerate}[(1)]
  \item\label{harary20} $R$ is bi-connected if and only if there exists a connected pair $(f, \, g)$ of two commuting maps
    $f$, $g: X \to X$ such that $R = R_{(f, \, g)}$, i.e. $R (x, \, y) = (f(x), \, g(y))$, for all $x$, $y\in X$.

  \item $R$ is ideal-simple in the sense of Definition \ref{def:solsimple} if and only if there exists an incompressible pair $(f,\, g)$ (Definition \ref{def:simple}(\ref{item:incomprdef}))
    of commuting maps $f$, $g: X\to X$ such that $R = R_{(f, \, g)}$. Thus, they are fully classified by Corollary \ref{cor:simpbimag} and Theorem \ref{cor:simpbimagb}(\ref{simp12}).
  \end{enumerate}
\end{theorem}

\begin{proof}
  \begin{enumerate}[(1)]
  \item The proof follows from the structure of P{\l}onka bi-magmas as given by Theorem \ref{th:strpbimag}. Indeed, since the multiplications $\cdot$ and $\ast$ on $X$ are given by \eqref{eq:multRPmdd}, we obtain that $R (X_i \times X_i) \subseteq X_i \times X_i$, for all $i\in I$. Now Definition \ref{def:connfunc}(\ref{item:bi-conn}) of bi-connected maps implies that the index set $I$ of the decomposition $X = \coprod_{i\in I} \, X_i$ has to be a singleton. Finally, the bi-connectedness of $R_{(f,g)}$ is equivalent to connectedness of $(f,g)$ by Remarks~\ref{re:biconevsinde}(\ref{bic_self}).

  \item Using Remark \ref{re:simplepov}(\ref{item:draci}) we obtain that $R$ is an ideal-simple solution if and only if $(X, \, \cdot, \, \ast)$ is a right-left ideal-simple P{\l}onka bi-magma. Theorem~\ref{cor:simpbimagb} now finishes the proof.  
\end{enumerate}
\end{proof}

\begin{proposition}\label{pr:biconn_indec}
Let $(X, \, \cdot, \, \ast)$ be a P{\l}onka bi-magma and $R = R_{(\cdot, \, \ast)}$ the associated YB-solution. The following conditions are equivalent:
\begin{enumerate}[(1)]
\item\label{item:bc1} $R$ is bi-connected;
\item\label{item:bc2} $R$ is indecomposable, i.e. there is no partition $X = Y \sqcup Z$ into non-empty subsets such that $R(Y\times Y)\subseteq Y\times Y$ and $R(Z\times Z)\subseteq Z\times Z$;
\item\label{item:bc3} the finest P{\l}onka bi-partition of $(X, \, \cdot, \, \ast)$ (Theorem~\ref{th:biuniqpart}) consists of a single block;
\item\label{item:bc4} $R = R_{(f, \, g)}$ for a connected pair $(f, \, g)$ of commuting maps $f$, $g: X\to X$.
\end{enumerate}
\end{proposition}
\begin{proof}
(\ref{item:bc1}) $\Rightarrow$ (\ref{item:bc2}) follows trivially by observing that a decomposition $X = Y\sqcup Z$ as in (\ref{item:bc2}) is a non-trivial partition in the sense of Definition~\ref{def:connfunc}(\ref{item:bi-conn}).

(\ref{item:bc2}) $\Rightarrow$ (\ref{item:bc3}). Let $(X_i,\ f_i^j,\ g_i^j)$ be an arbitrary P{\l}onka bi-partition of $(X, \, \cdot, \, \ast)$. By \eqref{eq:multRPmdd} we have $R(X_i\times X_j)\subseteq X_i\times X_j$ for all $i$, $j\in I$ and therefore $R(S\times S)\subseteq S\times S$ for every union $S$ of blocks $X_i$. If $|I|\geq 2$, then for any $i_0\in I$ the subsets $Y := X_{i_0}$ and $Z := \bigcup_{i\neq i_0} X_i$ would give a decomposition as in (\ref{item:bc2}). Thus, if $R$ is indecomposable, every P{\l}onka bi-partition of $(X, \, \cdot, \, \ast)$, and in particular the finest one, consists of a single block.

(\ref{item:bc3}) $\Rightarrow$ (\ref{item:bc4}). If the finest P{\l}onka bi-partition is $\{X\}$, then so is the coarsest one, since it is coarser. By Remark~\ref{re:Lybneb} we obtain $(X, \, \cdot, \, \ast) = (X, \, \cdot_f, \, \ast^g)$, i.e. $R = R_{(f, \, g)}$, for the commuting maps $f := f_1^1$ and $g := g_1^1$. Moreover, by Theorem~\ref{th:biuniqpart}(2) the family $\{f, \, g\}$ attached to the unique block of the finest P{\l}onka bi-partition is connected.

(\ref{item:bc4}) $\Rightarrow$ (\ref{item:bc1}) follows from Remarks~\ref{re:biconevsinde}(\ref{bic_self}), since $R_{(f, \, g)} = f\times g$.
\end{proof}

\begin{corollary}\label{cor:famI2}
  Let $(X, \, \cdot)$ be a right P{\l}onka magma and consider the associated YB-solution $R = R_{(\cdot, \, \cdot_{\rm op})} : X\times X \to X\times X$, $R (x, \, y) = (x\cdot y, \, y\cdot x)
  $. Then:
  \begin{enumerate}[(1)]
  \item $R$ is bi-connected if and only if there exists a connected map $f: X \to X$ such that $R(x, \, y) = (f(x), \, f(y))$, for all $x$, $y\in X$.
    In particular, the number of isomorphism classes of bi-connected solutions of the form \eqref{eq:famII} associated to right P{\l}onka magmas of order $n$ is
    the Harary number $\mathfrak{c} (n)$ \cite[Definition 3.6]{acm1}.

  \item $R$ is ideal-simple in the sense of Definition \ref{def:solsimple} if and only if $X$ is a finite set and there exists a cycle
    $f: X \to X$ of length $|X|$ such that $R (x, \, y) = (f(x), \, f(y))$, for all $x$, $y\in X$. Any two such ideal-simple solutions on sets of the same (finite) cardinality are isomorphic.
  \end{enumerate}
\end{corollary}

\begin{proof}
For (1) we use first the structure theorem of right P{\l}onka magmas (Theorem \ref{th:structRPm}) since in the decomposition $X = \coprod_{i\in I} \, X_i$ given by Theorem \ref{th:structRPm} we have $R (X_i \times X_i) \subseteq X_i \times X_i$, for all $i\in I$. Thus, $R$ is bi-connected if and only if $|I| = 1$ and there exists a connected map $f: X \to X$ such that $x\cdot y = f(x)$, for all $x$, $y\in X$. Furthermore, it can be easily seen that two solutions $R_{(f,f)}$ and $R_{(f',f')}$ are isomorphic if and only if $f$ and $f'$ are conjugate. Therefore the number of isomorphism classes of bi-connected solutions of the form \eqref{eq:famII} is the number of conjugacy classes of connected self-maps. The last statement of (1) now follows from \cite[Definition 3.6]{acm1}.

As for (2) we observe that $I$ is an ideal of $R = R_{(\cdot, \, \cdot_{\rm op})}$ as defined in Definition \ref{def:solsimple} if and only if $I$ is a right ideal in the right P{\l}onka magma $(X, \cdot)$. Thus, $R$ is an ideal-simple solution if and only if the right P{\l}onka magma $(X, \cdot)$ is right ideal-simple.  Now we just apply Theorem \ref{th:simplestr} and this finishes the proof.
\end{proof}



\addcontentsline{toc}{section}{References}

\Addresses

\end{document}